\documentclass[journal]{IEEEtran}

\newtheorem{theorem}{Theorem}[section]
\newtheorem{definition}{Definition}
\newtheorem{lemma}{Lemma}
\newtheorem{remark}{Remark}

\newtheorem{proposition}{Proposition}

\usepackage{bm}
\usepackage{cite}
\usepackage{graphicx}
\usepackage{epsfig}
\usepackage{psfrag}
\usepackage{amssymb,amsmath}
\usepackage{enumerate}
\usepackage{float}
\usepackage{multirow}

\begin{document}

\title{The Theory of Quaternion Matrix Derivatives}
\author{Dongpo~Xu and Danilo P. Mandic, \textit{Fellow}, \textit{IEEE}
\thanks{This work was submitted to IEEE TSP on Jun 03, 2014.}
\thanks{Dongpo Xu is with the College of Science, Harbin Engineering University, Harbin 150001, China, and with
the Department of Electrical and Electronic Engineering, Imperial College London, London SW7 2AZ, UK. (e-mail:dongpoxu@gmail.com)}
\thanks{Danilo P. Mandic is with the Department of Electrical and Electronic Engineering,
Imperial College London, London SW7 2AZ, UK. (e-mail: d.mandic@imperial.ac.uk).}}

\markboth{manuscript for IEEE Trans. Signal Process.}{Shell
\MakeLowercase{\textit{et al.}}: Bare Demo of IEEEtran.cls for
Journals}

\maketitle

\begin{abstract}
A systematic theory is introduced for calculating the derivatives of quaternion matrix function with respect to quaternion matrix variables.
The proposed methodology is equipped with the matrix product rule and chain rule and it is able to handle both analytic and nonanalytic functions.
This corrects a flaw in the existing methods, that is, the incorrect use of the traditional product rule. In the framework introduced,
the derivatives of quaternion matrix functions can be calculated directly without the differential of this function.
Key results are summarized in tables. Several examples show how the quaternion matrix derivatives can be used as an
important tool for solving problems related to signal processing.
\end{abstract}

\begin{keywords}
Quaternion differentials, quaternion matrix derivatives, Jacobian, non-analytic functions, GHR calculus.
\end{keywords}

\IEEEpeerreviewmaketitle

\section{Introduction}
In recent years, quaternion signal processing has attracted
considerable research interest in areas, including image processing
\cite{Pei,Moxey,Bulow,Felsberg,Ell07}, computer graphics \cite{Hanson}, aerospace and satellite tracking \cite{Kuipers,Fortuna}, modeling of wind profile \cite{Took09,Took10,Took102},
processing of polarized waves \cite{Bihan,Miron,Buchholz}, and design of space-time block codes \cite{Belfiore,Belfiore05,ChenLi,Hollanti,Sirianunpiboon,Seberry}.
Recent mathematical tools to support these developments include the quaternion
singular value decomposition \cite{Bihan}, quaternion Fourier transform \cite{Ell,Said},
statistical analysis \cite{Via10,Ginzberg,Took11} and Taylor series expansion \cite{Moreno12}.
However, gradient based optimisation techniques in quaternion algebra have experienced slow progress,
as the quaternion analyticity conditions are too stringent. For example, the generalised Cauchy-Riemann condition \cite{Watson} restricts the class of quaternion analytic
functions to linear functions and constants. One attempt to
relax this constraint is the so-called Cauchy-Riemann-Fueter (CRF) condition \cite{Sudbery}, however, the
even polynomial functions do not satisfy the CRF condition.
A slice regular condition was proposed in \cite{Gentili06,Gentili07}, which contains the polynomials and power series with one-sided quaternion
coefficients, however, the product and composition of two
slice regular functions are generally not slice regular.

In quaternion statistical signal processing,
a common optimization objective is to minimize a real cost function of quaternion variables, typically
in the form of error power, $f(q)=|e(q)|^2$, however, such a function is obviously not-analytic according to
quaternion analysis \cite{Sudbery,Deavours,Leo} and therefore quaternion derivative cannot be used.
To circumvent this problem, the so called pseudo-derivatives are often employed ,
which treat $f$ as a real analytic function of the four real components of quaternion variable, and
then take the real derivatives with respect to these independent real
parts, separately. However, using this approach is the computations become cumbersome and tedious even for very simple algorithms.
An alternative and more elegant approach that can deal with non-analytic functions directly in the quaternion domain is to use the HR calculus \cite{Mandic11}, which takes the derivatives
of $f$ with respect to quaternion variable and its involutions.
The HR calculus has been utilized to derive quaternion independent component analysis \cite{Javidi},
nonlinear adaptive filtering \cite{Ujang}, affine projection algorithms \cite{Jahanchahi13}, and Kalman filtering \cite{Jahanchahi14}.
However, the traditional product rule does not apply within the HR calculus
because of the non-commutativity of quaternion product.
The recently proposed generalized HR (GHR) calculus \cite{DPXU} rectifies this issue by making use of the quaternion
rotation. It also comprises a novel product rule and chain rule and bears a great resemblance to the CR (or Wirtinger) calculus \cite{Wirtinger,Brandwood,Kreutz},
which has been instrumental for the developments in complex-valued signal processing \cite{Lih08,Erdogan,Demissie,Bouboulis,Mandic2009b} and optimization \cite{Sorber}.
In \cite{DPXU}, the authors give a systematic treatment of the case of quaternion scalar functions which depend on quaternion argument, however,
the more general matrix case was not considered. Problems where the unknown parameter is a quaternion matrix are wide ranging including array signal processing \cite{Bihan,Miron}, space-time coding \cite{Belfiore05,ChenLi,Hollanti},
and quaternion orthogonal designs \cite{Seberry}.

The problem of finding derivatives with respect to real-valued matrices is well understood and has been studied in \cite{Magnus,Brewer,Harville}.
For the complex-valued vector case, the mathematical foundations for derivative operations have been considered in \cite{Brandwood},
where the major contribution is the notion of complex gradient and the condition of stationary point in the context of optimization.
This work was further extended to second order derivatives together with a duality relationship between the complex gradient and Hessian and their real bivariate counterparts \cite{Bos}.
A systematical treatment of all the related concepts has been summarized in \cite{Kreutz}.
More general complex matrix derivatives have been thoroughly studied in \cite{Hjorungnesbk}.

Our aim here is to establish a systematic theory for calculating the derivatives of matrix functions with respect to
quaternion matrix variables. To this end, the GHR calculus for scalars is used to develop calculus for functions of quaternion matrices,
The \textrm{vec} operator and the Jacobian matrix play an important role in the resulting calculus,
giving very general matrix product and chain rules.
In addition, the proposed rules are quite generic and can reduce to scalar calculus rules when the matrices involved are of order one.
For a real scalar function of quaternion matrix variable, the necessary conditions for optimality can be found by either setting the
derivative of the function with respect to the quaternion matrix variable or its quaternion involutions to zero.
Meanwhile, the direction of maximum rate of change of the function is given by the Hermitian of derivative of the function with respect to the quaternion matrix variable.
Our results offer therefore a generalization of the results for scalar functions of vector variables.
We generalize the complex-valued matrix derivatives given in \cite{Hjorungnes07,Hjorungnesbk} to the quaternion matrix case, and calculate directly the quaternion matrix derivatives without the quaternion differentials of the functions.
The proposed theory is useful for numerous optimization problems which involve quaternion matrix parameters.

The rest of this paper is organized as follows: Section II introduces
the basic concepts and properties of quaternion algebra.
A discussion of the differences between analytic and non-analytic functions is presented
in Section III. In Section IV, the quaternion differential is introduced and several key differentials are presented.
The definition and calculus rules of the quaternion matrix derivatives are given in Section V. Section VI contains some important results, such as
conditions for finding stationary points, the direction in which the
function has the maximum rate of change, and the steep descent method.
In Section VII, several key results
are placed in tables and some more practical results are derived based on proposed theory.
Finally, Section VIII concludes the paper.

\subsection{Notations}
In this paper, we use bold-faced upper case letters to denote
matrices, bold-faced lower case letters for column vectors, and
standard lower case letters for scalar quantities.
The classification of functions
and variables is shown in Table \ref{tb:funs}. Superscripts
$(\cdot)^*$, $(\cdot)^T$ and $(\cdot)^H$ denote the quaternion conjugate,
transpose and Hermitian (i.e., transpose and quaternion conjugate),
respectively. $\mathfrak{R}({\bm A})$, ${\rm Tr}({\bm A})$ and $\|{\bm A}\|$
denote the real part, trace and norm of ${\bm A}$, $\otimes$ and $\odot$ denote the Kronecker
and Hadamard product, ${\rm vec}(\cdot)$ vectorizes a matrix by stacking its columns, ${\bm I}_N$ is the identity matrix of dimension $N$, and ${\bm 0}_{N\times S}$ denotes
the $N\times S$ zero matrix. By ${\rm reshape}(\cdot)$ we refer to any linear reshaping operator of the matrix, examples of such operators are the transpose $(\cdot)^T$ and ${\rm vec}(\cdot)$.

\begin{table*}[!htbp]
  \centering
   \caption{Classification of functions and variables}\label{tb:funs}
\renewcommand\arraystretch{1.3}
\begin{tabular}{|c|c|c|c|}
\hline
 Function type  & Scalar variable $q\in\mathbb{H}$  & Vector variable ${\bm q}\in\mathbb{H}^{N\times 1}$ &Matrix variable ${\bm Q}\in\mathbb{H}^{N\times S}$  \\
\hline
\multicolumn{4}{c}{}\vspace{-6pt} \\
\hline
 Scalar function $f\in\mathbb{H}$ &$f(q)$&$f({\bm q})$& $f({\bm Q})$ \\
\hline
 Vector function ${\bm f}\in\mathbb{H}^{M\times 1}$ &${\bm f}(q)$&${\bm f}({\bm q})$& ${\bm f}({\bm Q})$ \\
\hline
 Matrix function ${\bm F}\in\mathbb{H}^{M\times P}$ &${\bm F}(q)$&${\bm F}({\bm q})$& ${\bm F}({\bm Q})$ \\
\hline
\end{tabular}\label{tab:truth2}
\end{table*}

\section{Quaternion Algebra}
Quaternions are an associative but not commutative algebra
over $\mathbb{R}$, defined as
\begin{equation}\label{eq:realreprent}
\mathbb{H}= \{q_a+iq_b+jq_c+kq_d\;|\;q_a,q_b,q_c,q_d\in\mathbb{R}\}
\end{equation}
where $\{1,i,j,k\}$ is a basis of $\mathbb{H}$, and the imaginary units $i,j$ and $k$ satisfy $i^2=j^2=k^2=ijk=-1$, which implies $ij=k=-ji$, $jk=i=-kj$, $ki=j=-ik$. For any quaternion
\begin{equation}\label{eq:qqabcd}
q=q_a+iq_b+jq_c+kq_d=Sq+Vq
\end{equation}
the scalar (real) part is denoted by $q_a=Sq=\mathfrak{R}(q)$, while the vector part $Vq=\mathfrak{I}(q)=iq_b+jq_c+kq_d$ comprises the three imaginary parts. The product for $p,q \in \mathbb{H}$ is given by
\begin{equation}
pq=SpSq-Vp \cdot Vq+SpVq+SqVp+Vp\times Vq
\end{equation}
where the symbols $'\cdot'$ and $'\times'$ denote the usual inner product and vector product, respectively.
The presence of the vector product means that the quaternion product is noncommutative, and in general for $p,q\in\mathbb{H}$, $pq\neq qp$. The conjugate of a quaternion $q$ is defined as $q^*=Sq-Vq$, while the conjugate of the product satisfies $(pq)^*=q^*p^*$. The modulus of a quaternion is defined as $|q|=\sqrt{qq^*}=\sqrt{q_a^2+q_b^2+q_c^2+q_d^2}$, and it is
easy to check that $|pq|=|p||q|$. The inner product of $p$ and $q$ is defined as $<p,q>=\mathfrak{R}(p^*q)$
The inverse of a quaternion $q\neq 0$ is $q^{-1}=q^*/|q|^2$, and an important property of the inverse is
\begin{equation}
(pq)^{-1}=\frac{(pq)^*}{|pq|^2}=\frac{q^*p^*}{|q|^2|p|^2}=\frac{q^*}{|q|^2}\frac{p^*}{|p|^2}=q^{-1}p^{-1}
\end{equation}
(note the change in order). If $|q| = 1$, we call $q$ a \textit{unit} quaternion. A quaternion $q$ is said to be \textit{pure} if $\mathfrak{R}(q)=0$, then
$q^*=-q$ and $q^2=-|q|^2$. Thus, a \textit{pure unit} quaternion is a square root of -1, such as the imaginary units $i,j$ and $k$.

Quaternions can also be written in the \textit{polar} form
\begin{equation}
q=|q|\left(\frac{Sq}{|q|}+\frac{Vq}{|Vq|}\frac{|Vq|}{|q|}\right)=|q|(\cos \theta + \hat{q} \sin \theta )
\end{equation}
where $\hat{q}=Vq/|Vq|$ is a pure unit quaternion and $\theta=\arccos(S_q/|q|)$ is the angle (or argument)
of the quaternion. We shall next introduce the quaternion rotation and involution.
\begin{definition}[Quaternion Rotation \cite{Ward}]\label{def:qrot}
For any quaternion $q$, consider the transformation
\begin{equation*}
 q^{\mu}\triangleq \mu q \mu^{-1}
\end{equation*}
where $\mu=|\mu|(\cos \theta + \hat{\mu} \sin \theta )$ is any non-zero quaternion.
This transformation geometrically describes a 3-dimensional rotation of the vector part of $q$ through an angle $2\theta$ about the vector part of $\mu$.
\end{definition}

Some basic properties of the notation in Definition \ref{def:qrot} (see \cite{DPXU,Buchholz}) are:
\begin{equation}\label{pr:pqmu}
(pq)^{\mu}=p^{\mu}q^{\mu},\;   pq=q^pp=q p^{(q^*)},\;   \forall p, q \in\mathbb{H}
\end{equation}
\begin{equation}\label{pr:def1qmunu}
q^{\mu \nu}=(q^{\nu} )^{\mu},\; q^{\mu*} \triangleq (q^*)^{\mu}=(q^{\mu})^*\triangleq q^{*\mu},\;  \forall \nu, \mu \in\mathbb{H}
\end{equation}
Note that the real representation in \eqref{eq:realreprent} can be
easily generalized to a general orthogonal system $\{1,i^{\mu},j^{\mu},k^{\mu} \}$ given in \cite{DPXU,Ward}, where the following properties hold
\begin{equation}\label{pr:prodijkmurl}
i^{\mu}i^{\mu}=j^{\mu}j^{\mu}=k^{\mu}k^{\mu}=i^{\mu}j^{\mu}k^{\mu}=-1
\end{equation}

\begin{definition}[Quaternion Involution \cite{Ell07}]\label{def:qinv}
The involution of a quaternion $q$ around a pure unit quaternion $\eta$ is
\begin{equation*}
q^{\eta}= \eta q \eta^{-1}=\eta q \eta^*=-\eta q \eta
\end{equation*}
and represents a rotation of $q$ about $\eta$ through $\pi$ .
\end{definition}

In particular, the involutions around the imaginary units $i,j,k$ are given by \cite{Ell07}
\begin{equation}\label{eq:invijk}
\begin{split}
&q^i=-iqi=q_a+iq_b-jq_c-kq_d\\
&q^j=-jqj=q_a-iq_b+jq_c-kq_d\\
&q^k=-kqk=q_a-iq_b-jq_c+kq_d
\end{split}
\end{equation}
which allows us to express the four real-valued components of a quaternion $q$ as \cite{Sudbery,Took11,Mandic11}
\begin{eqnarray}\label{eq:z1qlink}
q_a=\frac{1}{4}(q+q^i+q^j+q^k),\;q_b=\frac{1}{4i}(q+q^i-q^j-q^k)\\
q_c=\frac{1}{4j}(q-q^i+q^j-q^k),\;q_d=\frac{1}{4k}(q-q^i-q^j+q^k)
\end{eqnarray}
This is analogous to the complex case, where $x=\frac{1}{2}(z+z^*)$
and $y=-\frac{i}{2}(z-z^*)$ for any $z=x+iy$
\cite{Picinbono,Moreno08}. Note that the quaternion conjugation
operator $(\cdot)^*$ is also an involution, that is
\begin{equation}\label{eq:linkqconjq}
q^*=\frac{1}{2}(-q+q^i+q^j+q^k),\;q=\frac{1}{2}(-q^*+q^{i*}+q^{j*}+q^{k*})
\end{equation}

\section{Analytic versus Non-Analytic Functions}\label{sec:nonanaly}
A function that is analytic is also called regular, or monogenic. Due to the non-commutativity of quaternion products, there are two ways to write the quotient in the definition of quaternion derivative, as shown below.
\begin{proposition}[\cite{Gentili09}]\label{pr:tradefideri}
Let $D\subseteq \mathbb{H}$ be a simply-connected domain of definition of the function $f:D\rightarrow \mathbb{H}$. If for any $q\in D$
\begin{equation}\label{eq:leftlimderiv}
\lim_{h\rightarrow 0}[\left(f(q+h)-f(q)\right)h^{-1}]
\end{equation}
exists in $\mathbb{H}$, then necessarily $f(q)=\omega q +\lambda$ for some $\omega,\lambda \in \mathbb{H}$. If for any $q\in D$
\begin{equation}\label{eq:rtlimderiv}
\lim_{h\rightarrow 0}[h^{-1}\left(f(q+h)-f(q)\right)]
\end{equation}
exists in $\mathbb{H}$, then necessarily $f(q)=q \nu +\lambda$ for some $\nu,\lambda \in \mathbb{H}$.
\end{proposition}

Proposition \ref{pr:tradefideri} indicates that the traditional definitions of derivative in \eqref{eq:leftlimderiv} and \eqref{eq:rtlimderiv} are too restrictive.
One attempt to relax this constraint is the so-called Cauchy-Riemann-Fueter (CRF) equation, given by
\begin{equation}\label{eq:crf}
\frac{\partial f}{\partial q_a}+i\frac{\partial f}{\partial q_b}+j\frac{\partial f}{\partial q_c}+k\frac{\partial f}{\partial q_d}=0
\end{equation}
However, Gentili and Struppa in \cite{Gentili06,Gentili07} point out that the polynomial functions (even the identity $f(q)=q$) do not satisfy the CRF condition.
To further relax this constraint, a slice
monogenic condition was proposed in \cite{Gentili06}, by adopting the newer setting of slice domains to give
\begin{equation}\label{eq:lac}
\left(\frac{\partial}{\partial x}+I\frac{\partial}{\partial y}\right)f_I(x+Iy)=0,\; \textrm{for } \forall I\in \mathbb{S}
\end{equation}
where $q=x+Iy$ ($x, y$ real numbers), $f_I(q)$ is the restriction of $f(q)$ to the complex line $L_I=\mathbb{R}+I\mathbb{R}$ and $\mathbb{S}=\{I\in\mathbb{H}\;|\;I^2=-1\}$. This class of slice monogenic functions contains the polynomials (and, more generally, power series) with right-sided quaternion coefficients. However, the product and composition of two slice monogenic functions $f$ and $g$ are generally not slice monogenic. For example, if $g(q)=q$ and $f(q)=q\omega$, $\omega\in\mathbb{H}$, then $f$ and $g$ are slice monogenic functions, but the product $f(q)g(q)=q\omega q$ is not a slice monogenic function.

The quaternion derivative in quaternion analysis is defined only
for analytic functions. However, in engineering problems, objective functions of interest are often real-valued to minimize or maximize them
and thus not analytic, such as
\begin{equation}
f(q)=|q|^2=qq^*
\end{equation}
In order to take the derivative of such functions, the HR calculus extends the classical idea of complex CR calculus \cite{Brandwood,Kreutz,Wirtinger} to the quaternion field,
which comprises of two groups of derivatives: the HR-derivatives \cite{Mandic11}
\begin{align}\label{def:hrder}
{\large
\left[\begin{array}{cc}
\frac{\partial f}{\partial q}\\
\frac{\partial f}{\partial q^i}\\
\frac{\partial f}{\partial q^j}\\
\frac{\partial f}{\partial q^k}
\end{array}\right]}=\frac{1}{4}
\left[\begin{array}{cccc}
1 & -i  & -j & -k \\
1 & -i & j  & k\\
1 & i  & -j & k\\
1 & i  & j  & -k
\end{array}\right]
{\large
\left[\begin{array}{cc}
\frac{\partial f}{\partial q_a}\\
\frac{\partial f}{\partial q_b}\\
\frac{\partial f}{\partial q_c}\\
\frac{\partial f}{\partial q_d}
\end{array}\right]}
\end{align}
and the conjugate HR-derivatives
\begin{align}\label{def:hrconjder}
{\large
\left[\begin{array}{cc}
\frac{\partial f}{\partial q^*}\\
\frac{\partial f}{\partial q^{i*}}\\
\frac{\partial f}{\partial q^{j*}}\\
\frac{\partial f}{\partial q^{k*}}
\end{array}\right]}=\frac{1}{4}
\left[\begin{array}{cccc}
1 & i  & j & k \\
1 & i & -j  & -k\\
1 & -i  & j & -k\\
1 & -i  & -j  & k
\end{array}\right]
{\large
\left[\begin{array}{cc}
\frac{\partial f}{\partial q_a}\\
\frac{\partial f}{\partial q_b}\\
\frac{\partial f}{\partial q_c}\\
\frac{\partial f}{\partial q_d}
\end{array}\right]}
\end{align}
However, the traditional product rule is not valid for the HR calculus.
For example, $f(q)=|q|^2$, then $\frac{\partial |q|^2}{\partial q}=\frac{1}{2}q^*$ from \eqref{def:hrder},
but $ \frac{\partial |q|^2}{\partial q}\neq q\frac{\partial q^*}{\partial q}+\frac{\partial q}{\partial q}q^*=-\frac{1}{2}q+q^*$.
This difficulty has been solved within the framework of the GHR calculus, and show that it equips quaternion analysis with both the novel product rule and chain rule, see \cite{DPXU} for more details.

\begin{definition}[The GHR Derivatives \cite{DPXU}]\label{def:leftghr}
Let $q=q_a+iq_b+jq_c+kq_d$, where $q_a,q_b,q_c,q_d\in\mathbb{R}$. Then the left GHR derivatives, with respect to $q^{\mu}$ and $q^{\mu*}$ $(\mu\neq 0, \mu \in \mathbb{H})$ of the function $f$, are defined as
\begin{equation}\label{def:ghrder}
\begin{split}
\frac{\partial f}{\partial q^{\mu}}=\frac{1}{4}\left(\frac{\partial f}{\partial q_a}-\frac{\partial f}{\partial q_b}i^{\mu}-\frac{\partial f}{\partial q_c}j^{\mu}-\frac{\partial f}{\partial q_d}k^{\mu}\right)\\
\frac{\partial f}{\partial q^{\mu*}}=\frac{1}{4}\left(\frac{\partial f}{\partial q_a}+\frac{\partial f}{\partial q_b}i^{\mu}+\frac{\partial f}{\partial q_c}j^{\mu}+\frac{\partial f}{\partial q_d}k^{\mu}\right)
\end{split}
\end{equation}
and the right GHR derivatives are defined as
\begin{equation}
\begin{split}
\frac{\partial_r f}{\partial q^{\mu}}=\frac{1}{4}\left(\frac{\partial f}{\partial q_a}-i^{\mu}\frac{\partial f}{\partial q_b}-j^{\mu}\frac{\partial f}{\partial q_c}-k^{\mu}\frac{\partial f}{\partial q_d}\right)\\
\frac{\partial_r f}{\partial q^{\mu*}}=\frac{1}{4}\left(\frac{\partial f}{\partial q_a}+i^{\mu}\frac{\partial f}{\partial q_b}+j^{\mu}\frac{\partial f}{\partial q_c}+k^{\mu}\frac{\partial f}{\partial q_d}\right)
\end{split}
\end{equation}
where $\frac{\partial f}{\partial q_a}$, $\frac{\partial f}{\partial q_b}$, $\frac{\partial f}{\partial q_c}$ and $\frac{\partial f}{\partial q_d}$ are the partial derivatives of $f$ with respect to $q_a$, $q_b$, $q_c$ and $q_d$, respectively, and the set $\{1,i^{\mu},j^{\mu},k^{\mu}\}$ is a general orthogonal basis of $\mathbb{H}$.
\end{definition}

Some basic properties of the left GHR derivatives in Definition \ref{def:leftghr} (see \cite{DPXU}) are:
\begin{align}
&\textrm{Product rule}: \frac{\partial (fg)}{\partial q^{\mu}}=f\frac{\partial g}{\partial q^{\mu}}+\frac{\partial (fg)}{\partial q^{\mu}}\big|_{g=const} \label{rl:product}\\
&\textrm{Product rule}: \frac{\partial (fg)}{\partial q^{\mu*}}=f\frac{\partial g}{\partial q^{\mu*}}+\frac{\partial (fg)}{\partial q^{\mu*}}\big|_{g=const} \label{rl:product2}\\
&\textrm{Chain rule}:\frac{\partial f(g(q))}{\partial q^{\mu}}=\sum_{\nu\in\{1,i,j,k\}}\frac{\partial f}{\partial g^{\nu}}\frac{\partial g^{\nu}}{\partial q^{\mu}}\\
&\textrm{Chain rule}:\frac{\partial f(g(q))}{\partial q^{\mu*}}=\sum_{\nu\in\{1,i,j,k\}}\frac{\partial f}{\partial g^{\nu}}\frac{\partial g^{\nu}}{\partial q^{\mu*}}\\
&\textrm{Rotation rule} :   \left(\frac{\partial f}{\partial q^{\mu}}\right)^{\nu}
=\frac{\partial f^{\nu}}{\partial q^{\nu\mu}},\,\left(\frac{\partial f}{\partial q^{\mu*}}\right)^{\nu}
=\frac{\partial f^{\nu}}{\partial q^{\nu\mu*}}\\
&\quad \textrm{if $f$ is real } \left(\frac{\partial f}{\partial q^{\mu}}\right)^{\nu}
=\frac{\partial f}{\partial q^{\nu\mu}},\,\left(\frac{\partial f}{\partial q^{\mu*}}\right)^{\nu}
=\frac{\partial f}{\partial q^{\nu\mu*}}\label{rl:realrota}\\
&\textrm{Conjugate rule}:\left(\frac{\partial f}{\partial q^{\mu}}\right)^*
=\frac{\partial_r f^*}{\partial q^{\mu*}},\left(\frac{\partial f}{\partial q^{\mu*}}\right)^*=\frac{\partial_r f^*}{\partial q^{\mu}}\\
&\quad \textrm{if $f$ is real } \left(\frac{\partial f}{\partial q^{\mu}}\right)^*=\frac{\partial f}{\partial q^{\mu*}},\,
\left(\frac{\partial f}{\partial q^{\mu*}}\right)^*=\frac{\partial f}{\partial q^{\mu}}\label{rl:realconj}
\end{align}
\begin{remark}
Observe that for $\mu\in\{1,i,j,k\}$, the HR derivatives \eqref{def:hrder} and \eqref{def:hrconjder} are a special case of the right GHR derivative,
which is more concise and easier to understand.  Furthermore, the GHR derivatives
incorporate a novel product rule and chain rule, which is very convenient for calculating the GHR derivatives.
\end{remark}

\begin{remark}
Due to the non-commutativity of quaternion products, the left GHR derivative is different from the right GHR derivative. However, they will be equal if the function $f$ is real-valued. In the sequel, we mainly focus on the left GHR derivative, because it has a lot of convenient properties that are consistent with our common sense.
\end{remark}

\begin{table*}[!htbp]
  \centering
   \caption{Important results for quaternion matrix differentials}\label{tb:diffmat}
\renewcommand\arraystretch{1.3}
\begin{tabular}{|c|c|c|c|c|c|c|}
\hline
 Function  & ${\bm A}$   &$\alpha{\bm Q}\beta$&${\bm P}+{\bm Q}$ & $\textrm{Tr}({\bm Q})$& ${\bm P}{\bm Q}$ &${\bm P}\otimes{\bm Q}$ \\
\hline
 Differential &${\bm 0}$&$\alpha (d{\bm Q})\beta$&$d{\bm P}+d{\bm Q}$& $\textrm{Tr}(d{\bm Q})$&$(d{\bm P}){\bm Q}+{\bm P}(d{\bm Q})$ &$(d{\bm P})\otimes{\bm Q}+{\bm P}\otimes(d{\bm Q})$ \\
\hline
\multicolumn{7}{c}{}\vspace{-6pt} \\
\hline
Function  & ${\bm Q}^{\mu}$   &${\bm Q}^{\mu*}$&$\textrm{vec}({\bm Q})$ & $\textrm{reshape}({\bm Q})$& ${\bm Q}^{-1}$ &${\bm P}\odot{\bm Q}$ \\
\hline
Differential &$(d{\bm Q})^{\mu}$&$(d{\bm Q})^{\mu*}$
&$\textrm{vec}(d{\bm Q})$& $\textrm{reshape}(d{\bm Q})$&$-{\bm Q}^{-1}(d{\bm Q}){\bm Q}^{-1}$ &$(d{\bm P})\odot{\bm Q}+{\bm P}\odot(d{\bm Q})$ \\
\hline
\end{tabular}\label{tab:truth2}
\end{table*}

\section{Quaternion Differentials}

For a scalar function $f(q)$, where $q$ is a quaternion variable, the differential of $f(q)$ can be expressed as follows \cite{DPXU}
\begin{equation}\label{eq:scalarghr}
df =\sum_{\mu\in\{1,i,j,k\}}\frac{\partial f}{\partial q^{\mu}}dq^{\mu},\quad df =\sum_{\mu\in\{1,i,j,k\}}\frac{\partial f}{\partial q^{\mu*}}dq^{\mu*}
\end{equation}
where $\frac{\partial f}{\partial q^{\mu}}$ and $\frac{\partial f}{\partial q^{\mu*}}$ are the GHR derivatives defined in \eqref{def:ghrder}.

In like with this formula, we define the differential of an $M\times P$ matrix function ${\bm F}=[f_{mp}]$ to be
\begin{equation}
d{\bm F}\triangleq\left[\begin{array}{ccc}
df_{11}& \cdots &df_{1P}\\
\vdots  & \ddots & \vdots \\
df_{M1}& \cdots &df_{MP}
\end{array}\right]
\end{equation}
A procedure that can be used to find the differentials of ${\bm F}:\mathbb{H}^{N\times S}\rightarrow\mathbb{H}^{M\times P}$ is to calculate the difference
\begin{equation}\label{eq:diffF}
\begin{split}
{\bm F}({\bm Q}&+d{\bm Q})-{\bm F}({\bm Q})=\\
&\textrm{First-order}(d{\bm Q})+\textrm{Higher-order}(d{\bm Q})
\end{split}
\end{equation}
Then $d{\bm F}({\bm Q})=\textrm{First-order}(d{\bm Q})$, i.e., the first order terms of $d{\bm Q}$.
This definition complies with the multiplicative and associative rules
\begin{equation}\label{pr:mulassoc}
d(\alpha {\bm Q} \beta) = \alpha ( d{\bm Q})\beta,\quad d({\bm P}+{\bm Q}) = d{\bm P}+d{\bm Q}
\end{equation}
where $\alpha,\beta \in \mathbb{H}$. If ${\bm P}$ and ${\bm Q}$ are product-conforming matrices, it can be verified that the differential of their product
is
\begin{equation}
d({\bm P}{\bm Q}) = (d{\bm P}){\bm Q}+{\bm P}(d{\bm Q})
\end{equation}
Some of the most important results on quaternion matrix differentials are summarized in Table \ref{tb:diffmat}, assuming that ${\bm A}$, ${\bm B}$, and $\alpha,\beta$ to be quaternion constants, and
${\bm P},\;{\bm Q}$ to be quaternion matrix variables.
These results are a generalization of the complex matrix differentials found in \cite{Hjorungnes07} to the quaternion case.

\begin{definition}
The Moore-Penrose inverse of ${\bm Q}\in\mathbb{H}^{N\times S}$ is defined as a matrix ${\bm Q}^+\in\mathbb{H}^{S\times N}$ satisfying all of the following four criteria \cite{Horn}:
\begin{equation}
\begin{split}
({\bm Q}{\bm Q}^+)^H&={\bm Q}{\bm Q}^+,\quad ({\bm Q}^+{\bm Q})^H={\bm Q}^+{\bm Q}\\
{\bm Q}{\bm Q}^+{\bm Q}&={\bm Q},\quad {\bm Q}^+{\bm Q}{\bm Q}^+={\bm Q}^+
\end{split}
\end{equation}
where $(\cdot)^H$ denotes the Hermitian operator, or the quaternion conjugate transpose.
\end{definition}

\begin{proposition}\label{pro:dmpinver}
Let ${\bm Q}\in\mathbb{H}^{N\times S}$, then
\begin{equation}
\begin{split}
d{\bm Q}^+=-{\bm Q}^+(d{\bm Q}){\bm Q}^++{\bm Q}^+({\bm Q}^+)^H({\bm I}_N-{\bm Q}{\bm Q}^+)\\
+({\bm I}_S-{\bm Q}^+{\bm Q})(d{\bm Q}^H)({\bm Q}^+)^H{\bm Q}^+
\end{split}
\end{equation}
\end{proposition}
\begin{proof}
The proof of Proposition \ref{pro:dmpinver} is similar to that of Proposition 1 in \cite{Hjorungnes07}, so it is omitted here.
\end{proof}

The following lemma is very important for identifying the GHR
derivatives from the differential of a quaternion matrix function.
The quaternion components, that is, the real variables ${\bm q}_a,{\bm q}_b,{\bm q}_c$ and ${\bm q}_d$ are mutually
independent and hence so are their differentials. Although the quaternion variables ${\bm q},{\bm q}^i,{\bm q}^j$
and ${\bm q}^k$ are related, it is important to notice that their differentials are linearly independent in the following way.
\begin{lemma}\label{lem:uni}
Let ${\bm Q}\in \mathbb{H}^{N\times S}$ and let ${\bm A}_n\in \mathbb{H}^{M\times NS}$. If the normal case
\begin{equation}\label{eq:a1234q}
\begin{split}
{\bm A}_1&d\textrm{vec}({\bm Q})+{\bm A}_2d\textrm{vec}({\bm Q}^{i})\\
&+{\bm A}_3d\textrm{vec}({\bm Q}^{j})+{\bm A}_4d\textrm{vec}({\bm Q}^{k})={\bm 0}
\end{split}
\end{equation}
or the conjugate case
\begin{equation}
\begin{split}
{\bm A}_1&d\textrm{vec}({\bm Q}^{*})+{\bm A}_2d\textrm{vec}({\bm Q}^{i*})\\
&+{\bm A}_3d\textrm{vec}({\bm Q}^{j*})+{\bm A}_4d\textrm{vec}({\bm Q}^{k*})={\bm 0}
\end{split}
\end{equation}
for all $d{\bm Q}\in\mathbb{H}^{N\times S}$. Then ${\bm A}_n={\bm 0}_{M\times NS}$ for $n\in\{1,2,3,4\}$.
\end{lemma}
\begin{proof}
Let ${\bm Q}={\bm Q}_a+i{\bm Q}_b+j{\bm Q}_c+k{\bm Q}_d\in\mathbb{H}^{N\times S}$, where ${\bm Q}_a,{\bm Q}_b,{\bm Q}_c,{\bm Q}_d\in \mathbb{R}^{N\times S}$.
From \eqref{pr:pqmu} and \eqref{pr:mulassoc}, we have $d{\bm Q}=d{\bm Q}_a+i(d{\bm Q}_b)+j(d{\bm Q}_c)+k(d{\bm Q}_d)$,
$d{\bm Q}^{i}=d{\bm Q}_a+i(d{\bm Q}_b)-j(d{\bm Q}_c)-k(d{\bm Q}_d)$, $d{\bm Q}^{j}=d{\bm Q}_a-i(d{\bm Q}_b)+j(d{\bm Q}_c)-k(d{\bm Q}_d)$ and
$d{\bm Q}^{k}=d{\bm Q}_a-i(d{\bm Q}_b)-j(d{\bm Q}_c)+k(d{\bm Q}_d)$.
By substituting $d\textrm{vec}({\bm Q}),d\textrm{vec}({\bm Q}^{i}),d\textrm{vec}({\bm Q}^{j})$ and $d\textrm{vec}({\bm Q}^{k})$ into \eqref{eq:a1234q}, we have
$({\bm A}_1+{\bm A}_2+{\bm A}_3+{\bm A}_4)d\textrm{vec}({\bm Q}_a)+({\bm A}_1+{\bm A}_2-{\bm A}_3-{\bm A}_4)id\textrm{vec}({\bm Q}_b)
+({\bm A}_1-{\bm A}_2+{\bm A}_3-{\bm A}_4)jd\textrm{vec}({\bm Q}_c)+({\bm A}_1-{\bm A}_2-{\bm A}_3+{\bm A}_4)kd\textrm{vec}({\bm Q}_d)={\bm 0}$.
Since the differentials $d{\bm Q}_a,d{\bm Q}_b,d{\bm Q}_c$ and $d{\bm Q}_d$ are independent,
then ${\bm A}_1+{\bm A}_2+{\bm A}_3+{\bm A}_4={\bm 0}$, ${\bm A}_1+{\bm A}_2-{\bm A}_3-{\bm A}_4={\bm 0}$,
${\bm A}_1-{\bm A}_2+{\bm A}_3-{\bm A}_4={\bm 0}$ and ${\bm A}_1-{\bm A}_2-{\bm A}_3+{\bm A}_4={\bm 0}$.
Hence, it follows that ${\bm A}_1={\bm A}_2={\bm A}_3={\bm A}_4={\bm 0}$. The conjugate case can be proved in a similar way.
\end{proof}

\section{Definition of quaternion matrix derivatives}

\begin{table*}[!htbp]
  \centering
   \caption{Identification table for quaternion derivatives}\label{tb:identtb}
\renewcommand\arraystretch{1.3}
\begin{tabular}{|c|c|c|c|c|}
\hline
 Function type  & Differential & Derivatives wrt $q,{\bm q},{\bm Q}$ & Derivatives wrt $q^{*},{\bm q}^{*},{\bm Q}^{*}$ & Order of derivatives \\
\hline
\multicolumn{4}{c}{}\vspace{-6pt} \\
\hline
\multirow{2}{*}{$f(q)$} &$df=\sum_{\mu\in\{1,i,j,k\}}a_{\mu}dq^{\mu}$ &\multirow{2}{*}{$\mathcal{D}_{q}f(q)=a_1$}   &\multirow{2}{*}{$\mathcal{D}_{q^{*}}f(q)=b_1$} &\multirow{2}{*}{$1\times 1$} \\
   & $df=\sum_{\mu\in\{1,i,j,k\}}b_{\mu}dq^{\mu*}$& \multirow{2}{*}{} & \multirow{2}{*}{}& \multirow{2}{*}{} \\
\hline
\multirow{2}{*}{$f({\bm q})$} &$df=\sum_{\mu\in\{1,i,j,k\}}{\bm a}^T_{\mu}d{\bm q}^{\mu}$ &\multirow{2}{*}{$\mathcal{D}_{{\bm q}}f({\bm q})={\bm a}^T_1$}   &\multirow{2}{*}{$\mathcal{D}_{{\bm q}^{*}}f({\bm q})={\bm b}^T_1$} &\multirow{2}{*}{$1\times N$} \\
   & $df=\sum_{\mu\in\{1,i,j,k\}}{\bm b}^T_{\mu}d{\bm q}^{\mu*}$& \multirow{2}{*}{} & \multirow{2}{*}{}& \multirow{2}{*}{} \\
\hline
\multirow{2}{*}{$f({\bm Q})$} &$df=\sum_{\mu\in\{1,i,j,k\}}\textrm{vec}^T({\bm A}_{\mu})d\textrm{vec}({\bm Q}^{\mu})$ &\multirow{2}{*}{$\mathcal{D}_{{\bm Q}}f({\bm Q})=\textrm{vec}^T({\bm A}_1)$}   &\multirow{2}{*}{$\mathcal{D}_{{\bm Q}^{*}}f({\bm Q})=\textrm{vec}^T({\bm B}_1)$} &\multirow{2}{*}{$1\times NS$} \\
   & $df=\sum_{\mu\in\{1,i,j,k\}}\textrm{vec}^T({\bm B}_{\mu})d\textrm{vec}({\bm Q}^{\mu*})$& \multirow{2}{*}{} & \multirow{2}{*}{}& \multirow{2}{*}{} \\
\hline
\multirow{2}{*}{$f({\bm Q})$} &$df=\sum_{\mu\in\{1,i,j,k\}}\textrm{Tr}\{{\bm A}^T_{\mu}d{\bm Q}^{\mu}\}$ &\multirow{2}{*}{$\frac{\partial f}{\partial {\bm Q}}={\bm A}_1$}   &\multirow{2}{*}{$\frac{\partial f}{\partial {\bm Q}^{*}}={\bm B}_1$} &\multirow{2}{*}{$N\times S$} \\
   & $df=\sum_{\mu\in\{1,i,j,k\}}\textrm{Tr}\{{\bm B}^T_{\mu}d{\bm Q}^{*}\}$& \multirow{2}{*}{} & \multirow{2}{*}{}& \multirow{2}{*}{} \\
\hline
\multirow{2}{*}{${\bm f}(q)$} &$d{\bm f}=\sum_{\mu\in\{1,i,j,k\}}{\bm c}_{\mu}dq^{\mu}$ &\multirow{2}{*}{$\mathcal{D}_{q}{\bm f}(q)={\bm c}_1$}   &\multirow{2}{*}{$\mathcal{D}_{q^{*}}{\bm f}(q)={\bm d}_1$} &\multirow{2}{*}{$M\times 1$} \\
   & $d{\bm f}=\sum_{\mu\in\{1,i,j,k\}}{\bm d}_{\mu}dq^{\mu*}$& \multirow{2}{*}{} & \multirow{2}{*}{}& \multirow{2}{*}{} \\
\hline
\multirow{2}{*}{${\bm f}({\bm q})$} &$d{\bm f}=\sum_{\mu\in\{1,i,j,k\}}{\bm C}_{\mu}d{\bm q}^{\mu}$ &\multirow{2}{*}{$\mathcal{D}_{{\bm q}}{\bm f}({\bm q})={\bm C}_1$}   &\multirow{2}{*}{$\mathcal{D}_{{\bm q}^{*}}{\bm f}({\bm q})={\bm D}_1$} &\multirow{2}{*}{$M\times N$} \\
   & $d{\bm f}=\sum_{\mu\in\{1,i,j,k\}}{\bm D}_{\mu}d{\bm q}^{\mu*}$& \multirow{2}{*}{} & \multirow{2}{*}{}& \multirow{2}{*}{} \\
\hline
\multirow{2}{*}{${\bm f}({\bm Q})$} &$d{\bm f}=\sum_{\mu\in\{1,i,j,k\}}{\bm \alpha}_{\mu}d\textrm{vec}({\bm Q}^{\mu})$ &\multirow{2}{*}{$\mathcal{D}_{{\bm Q}}{\bm f}({\bm q})={\bm \alpha}_1$}   &\multirow{2}{*}{$\mathcal{D}_{{\bm Q}^{*}}{\bm f}({\bm q})={\bm \beta}_1$} &\multirow{2}{*}{$M\times NS$} \\
   & $d{\bm f}=\sum_{\mu\in\{1,i,j,k\}}{\bm \beta}_{\mu}d\textrm{vec}({\bm Q}^{\mu*})$& \multirow{2}{*}{} & \multirow{2}{*}{}& \multirow{2}{*}{} \\
\hline
\multirow{2}{*}{${\bm F}(q)$} &$d\textrm{vec}({\bm F})=\sum_{\mu\in\{1,i,j,k\}}{\bm g}_{\mu}dq^{\mu}$ &\multirow{2}{*}{$\mathcal{D}_{q}{\bm F}(q)={\bm g}_1$}   &\multirow{2}{*}{$\mathcal{D}_{q^{*}}{\bm F}(q)={\bm h}_1$} &\multirow{2}{*}{$MP\times 1$} \\
   & $d\textrm{vec}({\bm F})=\sum_{\mu\in\{1,i,j,k\}}{\bm h}_{\mu}dq^{\mu*}$& \multirow{2}{*}{} & \multirow{2}{*}{}& \multirow{2}{*}{} \\
\hline
\multirow{2}{*}{${\bm F}({\bm q})$} &$d\textrm{vec}({\bm F})=\sum_{\mu\in\{1,i,j,k\}}{\bm G}_{\mu}d{\bm q}^{\mu}$ &\multirow{2}{*}{$\mathcal{D}_{{\bm q}}{\bm F}({\bm q})={\bm G}_1$}   &\multirow{2}{*}{$\mathcal{D}_{{\bm q}^{*}}{\bm F}({\bm q})={\bm H}_1$} &\multirow{2}{*}{$MP\times N$} \\
   & $d\textrm{vec}({\bm F})=\sum_{\mu\in\{1,i,j,k\}}{\bm H}_{\mu}d{\bm q}^{\mu*}$& \multirow{2}{*}{} & \multirow{2}{*}{}& \multirow{2}{*}{} \\
\hline
\multirow{2}{*}{${\bm F}({\bm Q})$} &$d\textrm{vec}({\bm F})=\sum_{\mu\in\{1,i,j,k\}}{\bm \zeta}_{\mu}d\textrm{vec}({\bm Q}^{\mu})$ &\multirow{2}{*}{$\mathcal{D}_{{\bm Q}}{\bm F}({\bm Q})={\bm \zeta}_1$}   &\multirow{2}{*}{$\mathcal{D}_{{\bm Q}^{*}}{\bm F}({\bm Q})={\bm \xi}_1$} &\multirow{2}{*}{$MP\times NS$} \\
   & $d\textrm{vec}({\bm F})=\sum_{\mu\in\{1,i,j,k\}}{\bm \xi}_{\mu}d\textrm{vec}({\bm Q}^{\mu*})$& \multirow{2}{*}{} & \multirow{2}{*}{}& \multirow{2}{*}{} \\
\hline
\end{tabular}\label{tab:truth2}
\end{table*}

This section covers differentiation of matrix functions with respect to
a matrix variable ${\bm Q}$. Note that it is always assumed that all the elements of ${\bm Q}$ are linearly independent.

We start with a scalar function $f({\bm q})$ of an $N\times 1$ vector ${\bm q}$, then the GHR derivatives of $f$ are the $1\times N$ vector
\begin{align}
\frac{\partial f}{\partial {\bm q}^{\mu}}=\left[\frac{\partial f}{\partial q^{\mu}_1},\cdots,\frac{\partial f}{\partial q^{\mu}_N}\right]\\
\frac{\partial f}{\partial {\bm q}^{\mu*}}=\left[\frac{\partial f}{\partial q^{\mu*}_1},\cdots,\frac{\partial f}{\partial q^{\mu*}_N}\right]
\end{align}
The gradients of $f$ are the vector
\begin{equation}
\nabla_{{\bm q}^{\mu}} f\triangleq \left(\frac{\partial f}{\partial {\bm q}^{\mu}}\right)^T,\quad
\nabla_{{\bm q}^{\mu*}} f\triangleq \left(\frac{\partial f}{\partial {\bm q}^{\mu*}}\right)^T
\end{equation}
If ${\bm f}$ is an $M\times 1$ vector function of ${\bm q}$, then the $M\times N$ matrices
\begin{align}
\frac{\partial {\bm f}}{\partial {\bm q}^{\mu}}=\left[\begin{array}{cc}
\frac{\partial f_1}{\partial {\bm q}^{\mu}}\\
\vdots\\
\frac{\partial f_M}{\partial {\bm q}^{\mu}}
\end{array}\right],\quad
\frac{\partial {\bm f}}{\partial {\bm q}^{\mu*}}=\left[\begin{array}{cc}
\frac{\partial f_1}{\partial {\bm q}^{\mu*}}\\
\vdots\\
\frac{\partial f_M}{\partial {\bm q}^{\mu*}}
\end{array}\right]
\end{align}
are called the derivatives or Jacobian matrices of ${\bm f}$. Generalizing these concepts to matrix functions of matrices, we arrive at the following definition.

\begin{definition}\label{def:matghr}
Let ${\bm F}:\mathbb{H}^{N\times S}\rightarrow\mathbb{H}^{M\times P}$. Then the GHR derivatives (or Jacobian matrices) of ${\bm F}$ with respect to
${\bm Q}^{\mu},{\bm Q}^{\mu*}$ ($\mu\in\mathbb{H},\mu\neq0$) are the $MP\times NS$ matrices
\begin{align}
\mathcal{D}_{{\bm Q}^{\mu}}{\bm F}=\frac{\partial \textrm{vec}{\bm F}({\bm Q})}{\partial \textrm{vec}({\bm Q}^{\mu})}\label{def:matderi}\\
\mathcal{D}_{{\bm Q}^{\mu*}}{\bm F}=\frac{\partial \textrm{vec}{\bm F}({\bm Q})}{\partial \textrm{vec}({\bm Q}^{\mu*})}\label{def:matderi2}
\end{align}
The transposes of the Jacobian matrices  $\mathcal{D}_{{\bm Q}^{\mu}}{\bm F}$ and $\mathcal{D}_{{\bm Q}^{\mu*}}{\bm F}$ are called the gradients.
\end{definition}

Using the matrix derivative notations in Definition \ref{def:matghr}, the differentials of the scalar function $f$ in \eqref{eq:scalarghr} can be extended to the following matrix case
\begin{align}
d\textrm{vec}({\bm F})=\sum_{\mu\in\{1,i,j,k\}}(\mathcal{D}_{{\bm Q}^{\mu}}{\bm F})d\textrm{vec}({\bm Q}^{\mu}) \label{eq:dvecF}\\
d\textrm{vec}({\bm F})=\sum_{\mu\in\{1,i,j,k\}}(\mathcal{D}_{{\bm Q}^{\mu*}}{\bm F})d\textrm{vec}({\bm Q}^{\mu*})\label{eq:dvecF2}
\end{align}
This is also a generalization of the complex-valued matrix variable case studied thoroughly in \cite{Hjorungnesbk} to the case of the quaternion matrix variables.
It can be shown by using Lemma \ref{lem:uni} that the matrix derivatives in \eqref{eq:dvecF} and \eqref{eq:dvecF2} are unique.

Table \ref{tb:identtb} shows the connection between the differentials and derivatives of the different function types in Table \ref{tb:funs},
which is an extension of the corresponding table given in \cite{Hjorungnes07} for complex-valued variable case. In Table \ref{tb:identtb},
$q\in \mathbb{H}$, ${\bm q}\in \mathbb{H}^{N\times 1}$, ${\bm Q}\in \mathbb{H}^{N\times S}$,
$f\in \mathbb{H}$, ${\bm f}\in \mathbb{H}^{M\times 1}$ and ${\bm F}\in \mathbb{H}^{M\times P}$. Furthermore,
$a_{\mu},b_{\mu}\in\mathbb{H}$, ${\bm a}_{\mu},{\bm b}_{\mu}\in\mathbb{H}^{N\times 1}$, ${\bm A}_{\mu},{\bm B}_{\mu}\in\mathbb{H}^{N\times S}$,
${\bm c}_{\mu},{\bm d}_{\mu}\in\mathbb{H}^{M\times 1}$, ${\bm C}_{\mu},{\bm D}_{\mu}\in\mathbb{H}^{M\times N}$,
${\bm \alpha}_{\mu},{\bm \beta}_{\mu}\in\mathbb{H}^{M\times NS}$, ${\bm g}_{\mu},{\bm h}_{\mu}\in\mathbb{H}^{MP\times 1}$,
${\bm G}_{\mu},{\bm H}_{\mu}\in\mathbb{H}^{MP\times N}$, ${\bm \zeta}_{\mu},{\bm \xi}_{\mu}\in\mathbb{H}^{MP\times NS}$, and each of these might be a fucntion
of $q,{\bm q}$ or ${\bm Q}$.

\subsection{Product Rule}
In Section \ref{sec:nonanaly}, we have given a example to show that the traditional product rule is not valid for the HR  calculus.
Now, we generalize the product rules in \eqref{rl:product} and \eqref{rl:product2} for quaternion scalar variable to the case of quaternion matrix variable.
\begin{theorem}\label{thm:prodrl}
Let ${\bm H}:\mathbb{H}^{N\times S}\rightarrow \mathbb{H}^{M\times P}$ be given by ${\bm H}={\bm F}{\bm G}$, where
 ${\bm F}:\mathbb{H}^{N\times S}\rightarrow \mathbb{H}^{M\times R}$ and ${\bm G}:\mathbb{H}^{N\times S}\rightarrow \mathbb{H}^{R\times P}$. Then the following
relations hold
\begin{align}
&\mathcal{D}_{{\bm Q}^{\mu}}{\bm H}=({\bm I}_P\otimes {\bm F})\mathcal{D}_{{\bm Q}^{\mu}}{\bm G}+\mathcal{D}_{{\bm Q}^{\mu}}(\bm{FG})|_{{\bm G}=const}\label{eq:dmatfg}\\
&\mathcal{D}_{{\bm Q}^{\mu*}}{\bm H}=({\bm I}_P\otimes {\bm F})\mathcal{D}_{{\bm Q}^{\mu*}}{\bm G}+\mathcal{D}_{{\bm Q}^{\mu*}}(\bm{FG})|_{{\bm G}=const}\label{eq:dmatfg2}
\end{align}
\end{theorem}
\begin{proof}
The differential of ${\bm H}$ can be expressed as
\begin{equation}
d{\bm H}={\bm F}(d{\bm G}){\bm I}_P+d({\bm F}{\bm G})|_{{\bm G}=const}
\end{equation}
By using the differentials of ${\bm F}$ and ${\bm G}$ after applying the $\textrm{vec}(\cdot)$, we have
\begin{align}\label{eq:dfgproc}
d\textrm{vec}({\bm H})&=({\bm I}_P\otimes {\bm F})d\textrm{vec}({\bm G})+d\textrm{vec}({\bm F}{\bm G})|_{{\bm G}=const}\nonumber\\
&=\sum_{\mu\in\{1,i,j,k\}}({\bm I}_P\otimes {\bm F})\,(\mathcal{D}_{{\bm Q}^{\mu}}{\bm G})d\textrm{vec}({\bm Q}^{\mu})\nonumber\\
&\quad +\sum_{\mu\in\{1,i,j,k\}}\mathcal{D}_{{\bm Q}^{\mu}}(\bm{FG})|_{{\bm G}=const}\,d\textrm{vec}({\bm Q}^{\mu})\nonumber\\
&=\sum_{\mu\in\{1,i,j,k\}}\big[({\bm I}_P\otimes {\bm F})\mathcal{D}_{{\bm Q}^{\mu}}{\bm G}\nonumber\\
&\quad+\mathcal{D}_{{\bm Q}^{\mu}}(\bm{FG})|_{{\bm G}=const}\big]d\textrm{vec}({\bm Q}^{\mu})
\end{align}
where \eqref{pr:pqmu} and \eqref{pr:mulassoc} are used in the last equality.
Hence, the derivatives of $\bm{FG}$ with respect to ${\bm Q}^{\mu}$ can be identified as in \eqref{eq:dmatfg}. The second equality can be proved in similar manner.
\end{proof}

\subsection{Chain Rule}
One big advantage of the matrix derivatives are defined in Definition \ref{def:matghr} is that
the chain rule can be obtained in a very simple form, and is formulated in the following theorem.
\begin{theorem}\label{thm:nwpleftchainrl}
 Let $U\subseteq \mathbb{H}^{N\times S}$ and suppose ${\bm G}:U\rightarrow \mathbb{H}^{M\times P}$ has the GHR derivatives at an interior point ${\bm Q}$ of the set $U$.
 Let $V\subseteq \mathbb{H}^{M\times P}$ be such that ${\bm G}({\bm Q})\in V$ for all ${\bm Q} \in U$.
 Assume ${\bm F}:V\rightarrow \mathbb{H}^{R\times T}$ has GHR derivatives at an inner point ${\bm G}({\bm Q})\in V$,
 then the GHR derivatives of the composite function ${\bm H}({\bm Q})\triangleq{\bm F}({\bm G}({\bm Q}))$ are as follows:
\begin{align}
\mathcal{D}_{{\bm Q}^{\mu}}{\bm H}&=\sum_{\nu\in\{1,i,j,k\}}(\mathcal{D}_{{\bm G}^{\nu}}{\bm F})(\mathcal{D}_{{\bm Q}^{\mu}}{\bm G}^{\nu})\label{eq:dmatfgcomp}\\
\mathcal{D}_{{\bm Q}^{\mu*}}{\bm H}&=\sum_{\nu\in\{1,i,j,k\}}(\mathcal{D}_{{\bm G}^{\nu}}{\bm F})(\mathcal{D}_{{\bm Q}^{\mu*}}{\bm G}^{\nu})\\
\mathcal{D}_{{\bm Q}^{\mu}}{\bm H}&=\sum_{\nu\in\{1,i,j,k\}}(\mathcal{D}_{{\bm G}^{\nu*}}{\bm F})(\mathcal{D}_{{\bm Q}^{\mu}}{\bm G}^{\nu*})\\
\mathcal{D}_{{\bm Q}^{\mu*}}{\bm H}&=\sum_{\nu\in\{1,i,j,k\}}(\mathcal{D}_{{\bm G}^{\nu*}}{\bm F})(\mathcal{D}_{{\bm Q}^{\mu*}}{\bm G}^{\nu*})
\end{align}
\end{theorem}
\begin{proof}
From \eqref{eq:dvecF}, we have
\begin{equation}\label{eq:dvecHnu}
d\textrm{vec}({\bm H})=d\textrm{vec}({\bm F})=\sum_{\nu\in\{1,i,j,k\}}(\mathcal{D}_{{\bm G}^{\nu}}{\bm F})d\textrm{vec}({\bm G}^{\nu})
\end{equation}
The differential of  $d\textrm{vec}({\bm G}^{\nu})$ is given by
\begin{equation}\label{eq:dvecGnu}
d\textrm{vec}({\bm G}^{\nu})=\sum_{\mu\in\{1,i,j,k\}}(\mathcal{D}_{{\bm Q}^{\mu}}{\bm G}^{\nu})d\textrm{vec}({\bm Q}^{\mu})
\end{equation}
By substituting \eqref{eq:dvecGnu} into \eqref{eq:dvecHnu}, we have
\begin{equation}
d\textrm{vec}({\bm H})=\sum_{\mu,\nu\in\{1,i,j,k\}}(\mathcal{D}_{{\bm G}^{\nu}}{\bm F})(\mathcal{D}_{{\bm Q}^{\mu}}{\bm G}^{\nu})d\textrm{vec}({\bm Q}^{\mu})
\end{equation}
According to \eqref{eq:dvecF}, the derivatives of ${\bm H}$ with respect to ${\bm Q}^{\mu}$ can now be identified as in \eqref{eq:dmatfgcomp}. The
other equalities can be proved in similar manner.
\end{proof}

\section{Quaternion optimization theory}
In engineering, objective functions of interest are often real-valued and non-analytic, so it is important to find stationary points for scalar real-valued functions dependent on quaternion matrices
and the direction where such functions have maximum rates of change.

\subsection{Stationary Points}
This subsection shows that five equivalent ways can be used to find stationary points of $f({\bm Q})\in\mathbb{R}$,
which are the necessary conditions for optimality.

\begin{lemma}\label{lem:realconj}
Let $f:\mathbb{H}^{N\times S}\rightarrow \mathbb{R}$. Then the following holds
\begin{equation}\label{eq:realconjnu}
\left(\mathcal{D}_{{\bm Q}}f\right)^{\nu}=\mathcal{D}_{{\bm Q}^{\nu}}f,\quad \mathcal{D}_{{\bm Q}^{\nu*}}f=\left(\mathcal{D}_{{\bm Q}}f\right)^{\nu*}
\end{equation}
\end{lemma}
\begin{proof}
Using \eqref{rl:realrota} and \eqref{rl:realconj}, the lemma follows.
\end{proof}

\begin{theorem}
Let $f:\mathbb{H}^{N\times S}\rightarrow \mathbb{R}$, and let
${\bm Q}={\bm Q}_a+i{\bm Q}_b+j{\bm Q}_c+k{\bm Q}_d$, where ${\bm Q}_a,{\bm Q}_b,{\bm Q}_c,{\bm Q}_d\in \mathbb{R}^{N\times S}$.
A stationary point of the function $f({\bm Q})=g({\bm \xi})$ can be founded by one of the following five equivalent conditions
\begin{align}
&\mathcal{D}_{{\bm \xi}}g({\bm \xi})={\bm 0}\;\Leftrightarrow\; \mathcal{D}_{{\bm \zeta}}f({\bm Q})={\bm 0}\;\Leftrightarrow\; \mathcal{D}_{{\bm Q}}f({\bm Q})={\bm 0}\label{eq:equv1}\\
&\mathcal{D}_{{\bm \xi}}g({\bm \xi})={\bm 0}\;\Leftrightarrow\;\mathcal{D}_{{\bm \zeta}^*}f({\bm Q})={\bm 0}\;\Leftrightarrow\; \mathcal{D}_{{\bm Q}^*}f({\bm Q})={\bm 0}
\end{align}
where ${\bm \xi}=\left[{\bm Q}_a,{\bm Q}_b,{\bm Q}_c,{\bm Q}_d\right]$ and ${\bm \zeta}=\left[{\bm Q},{\bm Q}^i,{\bm Q}^j,{\bm Q}^k\right]$.
\end{theorem}

\begin{proof}
In \cite{Magnus}, a stationary point is defined as point where the derivatives of the function are equal to the null vector.
Thus, $\mathcal{D}_{{\bm \xi}}g({\bm \xi})={\bm 0}$ gives a stationary point by definition. Using the chain rule in \eqref{eq:dmatfgcomp} on both sides of $f({\bm Q})=g({\bm \xi})$, we obtain
\begin{align}\label{eq:dfgQxi}
\sum_{\nu\in\{1,i,j,k\}}(\mathcal{D}_{{\bm Q}^{\nu}}f)(\mathcal{D}_{{\bm \xi}}{\bm Q}^{\nu})=\mathcal{D}_{{\bm \xi}}g
\end{align}
From ${\bm Q}^{\nu}={\bm Q}_a+i^{\nu}{\bm Q}_b+j^{\nu}{\bm Q}_c+k^{\nu}{\bm Q}_d$, it follows that $\mathcal{D}_{{\bm \xi}}{\bm Q}^{\nu}=\left[{\bm I},i^{\nu}{\bm I},j^{\nu}{\bm I},k^{\nu}{\bm I}\right]$, where ${\bm I}\in\mathbb{R}^{NS\times NS}$ is the identity matrix. By substituting these results into \eqref{eq:dfgQxi}, we have
\begin{align}\label{eq:DfzetaJ}
(\mathcal{D}_{{\bm \zeta}}f){\bm J}=(\mathcal{D}_{{\bm \zeta}}f)\left[\begin{array}{cccc}
{\bm I} & i{\bm I}  & j{\bm I} & k{\bm I} \\
{\bm I} & i{\bm I}  & -j{\bm I} & -k{\bm I} \\
{\bm I} & -i{\bm I}  & j{\bm I} & -k{\bm I} \\
{\bm I} & -i{\bm I}  & -j{\bm I} & k{\bm I}
\end{array}\right] =\mathcal{D}_{{\bm \xi}}g
\end{align}
where $\mathcal{D}_{{\bm \zeta}}f=\left[\mathcal{D}_{{\bm Q}}f,\mathcal{D}_{{\bm Q}^i}f,\mathcal{D}_{{\bm Q}^j}f,\mathcal{D}_{{\bm Q}^k}f\right]$ and $4{\bm J}{\bm J}^H={\bm I}_{4NS}$.
From \eqref{eq:DfzetaJ}, the equalities in \eqref{eq:equv1} are equivalent. The other equivalent relations can be proved by Lemma \ref{lem:realconj}.
\end{proof}

\subsection{Direction of maximum rate of change}
The next theorem states how to find the maximum rate of change of $f({\bm Q})\in\mathbb{R}$, which is widely applied in steepest descent methods, such as quaternion adaptive filters.
\begin{theorem}\label{thm:extchang}
Let $f:\mathbb{H}^{N\times S}\rightarrow \mathbb{R}$. Then the gradient $[\mathcal{D}_{{\bm Q}^*}f({\bm Q})]^T=[\mathcal{D}_{{\bm Q}}f({\bm Q})]^H$ defines the direction of the maximum rate of change
of $f$ with respect to $\textrm{vec}({\bm Q})$.
\end{theorem}
\begin{proof}
From \eqref{eq:dvecF}, \eqref{eq:realconjnu}, we have
\begin{equation}\label{eq:dfmatnu}
\begin{split}
df &= \sum_{\mu\in\{1,i,j,k\}}(\mathcal{D}_{{\bm Q}^{\mu}}f)d\textrm{vec}({\bm Q}^{\mu})\\
&= \sum_{\mu\in\{1,i,j,k\}}(\mathcal{D}_{{\bm Q}}f)^{\mu}(d\textrm{vec}({\bm Q}))^{\mu}
\end{split}
\end{equation}
Using \eqref{pr:pqmu}, \eqref{eq:z1qlink} and \eqref{eq:realconjnu}, we can further get
\begin{equation}\label{eq:dfredqf}
\begin{split}
df& = \sum_{\mu\in\{1,i,j,k\}}\left[(\mathcal{D}_{{\bm Q}}f)d\textrm{vec}({\bm Q})\right]^{\mu}\\
&= 4\mathfrak{R}\left[(\mathcal{D}_{{\bm Q}}f)d\textrm{vec}({\bm Q})\right]= 4\mathfrak{R}\left[(\mathcal{D}_{{\bm Q}^*}f)^*d\textrm{vec}({\bm Q})\right]\\
&= 4\left<(\mathcal{D}_{{\bm Q}^*}f)^T,d\textrm{vec}({\bm Q})\right>
\end{split}
\end{equation}
where $<\cdot,\cdot>$ is the Euclidean inner product between real vector in $\mathbb{R}^{4NS\times 1}$.
By applying the Cauchy-Schwartz inequality to \eqref{eq:dfredqf}, we obtain
\begin{eqnarray}
|df|=4\left|\left<(\mathcal{D}_{{\bm Q}^*}f)^T,d\textrm{vec}({\bm Q})\right>\right|\leq4\left\|\mathcal{D}_{{\bm Q}}f\right\|\left\|d\textrm{vec}({\bm Q})\right\|
\end{eqnarray}
which indicates that the maximum change of $f$ occurs when $d\textrm{vec}({\bm Q})$ is in the direction of $(\mathcal{D}_{{\bm Q}^*}f)^T=(\mathcal{D}_{{\bm Q}}f)^H$ from \eqref{eq:realconjnu}.
Thus, the steepest descent method can be expressed as
\begin{equation}\label{eq:stdesc}
\textrm{vec}({\bm Q}_{n+1})=\textrm{vec}({\bm Q}_{n})-\eta(\mathcal{D}_{{\bm Q}^*}f({\bm Q}_{n}))^T,\quad f\in\mathbb{R}
\end{equation}
where $\eta>0$ is the step size, and ${\bm Q}_{n+1}$ is the value of the unknown
matrix after $n$ iterations.
\end{proof}

\begin{table*}[!htbp]
  \centering
\caption{Derivatives of functions of the type $f(q)$}\label{tb:fderiv}
\renewcommand{\arraystretch}{1.3} 
\begin{tabular}{|c|c|c|c|}
 \hline
$f(q)$    & $\mathcal{D}_{q}f$ & $\mathcal{D}_{q^{*}}f$ & Note  \\
 \hline
 \multicolumn{4}{c}{}\vspace{-6pt} \\
\hline
$q$  & $1$  &$-\frac{1}{2}$& $--$ \\
 \hline
$\alpha q$  & $\alpha$  &$-\frac{1}{2}\alpha$& $\forall\alpha\in\mathbb{H}$ \\
 \hline
$q\beta$  & $\mathfrak{R}(\beta)$  &$-\frac{1}{2}\beta^*$&$\forall\beta\in\mathbb{H}$ \\
 \hline
$\alpha q\beta+\lambda$  & $\alpha\mathfrak{R}(\beta)$  &$-\frac{1}{2}\alpha\beta^*$&$\forall\alpha,\beta,\lambda\in\mathbb{H}$ \\
 \hline
$q^*$  & $-\frac{1}{2}$  &$1$ &$--$ \\
 \hline
$\alpha q^*$  & $-\frac{1}{2}\alpha$  &$\alpha$ &$\forall\alpha\in\mathbb{H}$ \\
 \hline
$q^*\beta$  & $-\frac{1}{2}\beta^*$  &$\mathfrak{R}(\beta)$& $\forall\beta\in\mathbb{H}$ \\
 \hline
$\alpha q^*\beta+\lambda$  & $-\frac{1}{2}\alpha\beta^*$  &$\alpha\mathfrak{R}(\beta)$& $\forall\alpha,\beta,\lambda\in\mathbb{H}$ \\
 \hline
 $q\alpha q\beta$  & $q\alpha\mathfrak{R}(\beta)+\mathfrak{R}(\alpha q\beta)$  &$-\frac{1}{2}q\alpha\beta^*-\frac{1}{2}(\alpha q\beta)^*$& $\forall\alpha,\beta\in\mathbb{H}$ \\
 \hline
 $q\alpha q^*\beta$  & $-\frac{1}{2}q\alpha\beta^*+\mathfrak{R}(\alpha q^*\beta)$  &$q\alpha\mathfrak{R}(\beta)-\frac{1}{2}(\alpha q^*\beta)^*$& $\forall\alpha,\beta\in\mathbb{H}$ \\
 \hline
 $q^*\alpha q\beta$  & $q^*\alpha\mathfrak{\beta}-\frac{1}{2}(\alpha q\beta)^*$  &$-\frac{1}{2}q^*\alpha\beta^*+\mathfrak{R}(\alpha q\beta)$& $\forall\alpha,\beta\in\mathbb{H}$ \\
 \hline
 $q^*\alpha q^*\beta$  & $-\frac{1}{2}q^*\alpha\beta^*-\frac{1}{2}(\alpha q^*\beta)^*$  &$q^*\alpha\mathfrak{R}(\beta)+\mathfrak{R}(\alpha q^*\beta)$& $\forall\alpha,\beta\in\mathbb{H}$ \\
 \hline
 $|V_q|$  & $-\frac{1}{4} \frac{V_q}{|V_q|}$  &$\frac{1}{4} \frac{V_q}{|V_q|}$ &$V_q=\mathfrak{I}(q)$  \\
 \hline
  $\frac{V_q}{|V_q|}$  & $\frac{1}{2|V_q|}$  &$-\frac{1}{2|V_q|}$ &$--$ \\
  \hline
 $\arctan\left(\frac{|V_q|}{S_q}\right)$  & $-\frac{\hat{q}q^*}{4|q|^2}$  &$\frac{\hat{q}q}{4|q|^2}$ &$S_q=\mathfrak{R}(q)$ \\
 \hline
 $q^{-1}$  & $-q^{-1}\mathfrak{R}(q^{-1})$  &$\frac{1}{2|q|^2}$ &$--$ \\
 \hline
$(q^*)^{-1}$  & $\frac{1}{2|q|^2}$  &$-f\mathfrak{R}(f)$ &$--$ \\
 \hline
$(\alpha q \beta+\lambda)^{-1}$  & $-f\alpha\mathfrak{R}(\beta f)$  &$\frac{1}{2}f\alpha(\beta f)^*$ &$\forall\alpha,\beta,\lambda\in\mathbb{H}$ \\
 \hline
$(\alpha q^*\beta+\lambda)^{-1}$  & $\frac{1}{2}f\alpha(\beta f)^*$  &$-f\alpha\mathfrak{R}(\beta f)$ & $\forall\alpha,\beta,\lambda\in\mathbb{H}$ \\
 \hline
$q^2$  & $q+\mathfrak{R}( q)$  &$-\frac{1}{2}q-\frac{1}{2}( q)^*$& $--$  \\
 \hline
$(q^*)^2$  & $-\frac{1}{2}q^*-\frac{1}{2}(q^*)^*$  &$q^*+\mathfrak{R}(q^*)$ & $--$\\
 \hline
$(\alpha q\beta+\lambda)^2$  & $g\alpha\mathfrak{R}(\beta)+\alpha\mathfrak{R}(\beta g)$  &$-\frac{1}{2}g\alpha\beta^*-\frac{1}{2}\alpha(\beta g)^*$& $g=\alpha q\beta+\lambda$ \\
 \hline
$(\alpha q^*\beta+\lambda)^2$  & $-\frac{1}{2}g\alpha\beta^*-\frac{1}{2}\alpha(\beta g)^*$  &$g\alpha\mathfrak{R}(\beta)+\alpha\mathfrak{R}(\beta g)$ & $g=\alpha q^*\beta+\lambda$\\
 \hline
$\mathfrak{R}(q)$  & $\frac{1}{4}$  &$\frac{1}{4}$ & $--$\\
 \hline
 $\mathfrak{R}(\alpha q\beta+\lambda)$  & $\frac{1}{4}\beta\alpha$  &$\frac{1}{4}\alpha^*\beta^*$ & $\forall\alpha,\beta,\lambda\in\mathbb{H}$\\
 \hline
$\mathfrak{R}(\alpha q^*\beta+\lambda)$  & $\frac{1}{4}\alpha^*\beta^*$ & $\frac{1}{4}\beta\alpha$ & $\forall\alpha,\beta,\lambda\in\mathbb{H}$\\
 \hline
$\frac{q}{|q|}$  & $\frac{3}{4|q|}$  &$-\frac{1}{2|q|}-\frac{1}{4|q|^3}q^2$ & $--$\\
 \hline
$\frac{q^*}{|q|}$  & $-\frac{1}{2|q|}-\frac{1}{4|q|^3}(q^*)^2$  &$\frac{3}{4|q|}$ & $--$\\
 \hline
$\frac{\alpha q\beta+\lambda}{|\alpha q\beta+\lambda|}$  & $\frac{\alpha}{2|g|}\mathfrak{R}(\beta)+\frac{g}{4|g|^3}\beta^*(\alpha^*g)^*$  &$-\frac{\alpha}{4|g|}\beta^*-\frac{g}{2|g|^3}\beta^*\mathfrak{R}(\alpha^*g)$& $g=\alpha q\beta+\lambda$ \\
 \hline
$\frac{\alpha q^*\beta+\lambda}{|\alpha q^*\beta+\lambda|}$  & $-\frac{\alpha}{2|g|}\beta^*-\frac{f}{|g|}\frac{\partial |g|}{\partial q^{}}$  &$\frac{\alpha}{|g|}\mathfrak{R}(\beta)-\frac{f}{|g|}\frac{\partial |g|}{\partial q^{*}}$ & $g=\alpha q^*\beta+\lambda$ \\
 \hline
$|q|$  & $\frac{1}{4|q|} q^*$  &$\frac{1}{4|q|} q$ & $--$\\
 \hline
$|q|^2$  & $\frac{1}{2} q^*$  &$\frac{1}{2} q$ & $--$\\
 \hline
$|\alpha q\beta+\lambda|$  & $\frac{g^*}{2|g|}\alpha\mathfrak{R}(\beta)-\frac{1}{4|g|}\beta^*(\alpha^*g)^*$  &$-\frac{g^*}{4|g|}\alpha\beta^*+\frac{1}{2|g|}\beta^*\mathfrak{R}(\alpha^*g)$& $g=\alpha q\beta+\lambda$ \\
 \hline
$|\alpha q^*\beta+\lambda|$  & $\frac{g}{2|g|}\beta^*\mathfrak{R}(\alpha^*)-\frac{1}{4|g|}\alpha(\beta g^*)^*$  &$-\frac{g}{4|g|}\beta^*(\alpha^*)^*+\frac{1}{2|g|}\alpha\mathfrak{R}(\beta g^*)$& $g=\alpha q^*\beta+\lambda$ \\
 \hline
$|\alpha q\beta+\lambda|^2$  & $g^*\alpha\mathfrak{R}(\beta)-\frac{1}{2}\beta^*(\alpha^*g)^*$  &$-\frac{1}{2}g^*\alpha\beta^*+\beta^*\mathfrak{R}(\alpha^*g)$& $g=\alpha q\beta+\lambda$ \\
 \hline
$|\alpha q^*\beta+\lambda|^2$  & $g\beta^*\mathfrak{R}(\alpha^*)-\frac{1}{2}\alpha(\beta g^*)^*$  &$-\frac{1}{2}g\beta^*(\alpha^*)^*+\alpha\mathfrak{R}(\beta g^*)$& $g=\alpha q^*\beta+\lambda$ \\
\hline
\end{tabular}\\[6pt]
\end{table*}

\begin{table*}[!htbp]
  \centering
\caption{Derivatives of functions $f({\bm q})$ and ${\bm f}({\bm q})$}\label{tb:derivfbmf}
\renewcommand{\arraystretch}{1.3} 
\begin{tabular}{|c|c|c|c|}
 \hline
$f({\bm q})$    & $ \mathcal{D}_{\bm q}f$ & $\mathcal{D}_{{\bm q}^{*}}f$ & Note \\
 \hline
 \multicolumn{4}{c}{}\vspace{-6pt} \\
\hline
 ${\bm a}^T{\bm q}\beta$  & ${\bm a}^T\mathfrak{R}(\beta)$  &$-\frac{1}{2}{\bm a}^T\beta^*$ & $\forall \bm{a}\in\mathbb{H}^{N\times 1},\beta\in\mathbb{H}$\\
 \hline
 ${\bm a}^T{\bm q}^*\beta$  & $-\frac{1}{2}{\bm a}^T\beta^*$  &${\bm a}^T\mathfrak{R}(\beta)$ & $\forall\bm{a}\in\mathbb{H}^{N\times 1},\beta\in\mathbb{H}$\\
 \hline
$\alpha{\bm q}^T{\bm b}$  & $\alpha\mathfrak{R}({\bm b}^T)$  &$-\frac{1}{2}\alpha{\bm b}^H$ & $\forall\bm{b}\in\mathbb{H}^{N\times 1},\alpha\in\mathbb{H}$\\
 \hline
$\alpha{\bm q}^H{\bm b}$  & $-\frac{1}{2}\alpha{\bm b}^H$  &$\alpha\mathfrak{R}({\bm b}^T)$ & $\forall\bm{b}\in\mathbb{H}^{N\times 1},\alpha\in\mathbb{H}$\\
 \hline
${\bm a}^T{\bm q}\alpha{\bm q}^T{\bm b}$  & ${\bm a}^T\mathfrak{R}(\alpha{\bm q}^T{\bm b})+{\bm a}^T{\bm q}\alpha\mathfrak{R}({\bm b}^T)$
&$-\frac{1}{2}{\bm a}^T(\alpha{\bm q}^T{\bm b})^*-\frac{1}{2}{\bm a}^T{\bm q}\alpha{\bm b}^H$ & $\forall{\bm a},{\bm b}\in\mathbb{H}^{N\times 1},\alpha\in\mathbb{H}$\\
 \hline
${\bm a}^T{\bm q}\alpha{\bm q}^H{\bm b}$  & ${\bm a}^T\mathfrak{R}(\alpha{\bm q}^H{\bm b})-\frac{1}{2}{\bm a}^T{\bm q}\alpha{\bm b}^H$
&$-\frac{1}{2}{\bm a}^T(\alpha{\bm q}^H{\bm b})^*+{\bm a}^T{\bm q}\alpha\mathfrak{R}({\bm b}^T)$ & $\forall{\bm a},{\bm b}\in\mathbb{H}^{N\times 1},\alpha\in\mathbb{H}$\\
 \hline
 ${\bm a}^T{\bm q}^*\alpha{\bm q}^T{\bm b}$  & $-\frac{1}{2}{\bm a}^T(\alpha{\bm q}^T{\bm b})^*+{\bm a}^T{\bm q}^*\alpha\mathfrak{R}({\bm b}^T)$
&${\bm a}^T\mathfrak{R}(\alpha{\bm q}^T{\bm b})-\frac{1}{2}{\bm a}^T{\bm q}^*\alpha{\bm b}^H$ & $\forall{\bm a},{\bm b}\in\mathbb{H}^{N\times 1},\alpha\in\mathbb{H}$\\
 \hline
${\bm a}^T{\bm q}^*\alpha{\bm q}^H{\bm b}$  & $-\frac{1}{2}{\bm a}^T(\alpha{\bm q}^H{\bm b})^*-\frac{1}{2}{\bm a}^T{\bm q}^*\alpha{\bm b}^H$
&${\bm a}^T\mathfrak{R}(\alpha{\bm q}^H{\bm b})+{\bm a}^T{\bm q}^*\alpha\mathfrak{R}({\bm b}^T)$ & $\forall{\bm a},{\bm b}\in\mathbb{H}^{N\times 1},\alpha\in\mathbb{H}$\\
 \hline
 ${\bm q}^T{\bm A}{\bm q}$ & ${\bm q}^T{\bm A}+\mathfrak{R}((\bm{Aq})^T)$
 &$-\frac{1}{2}{\bm q}^T{\bm A}-\frac{1}{2}(\bm{Aq})^H$ & $\forall\bm{A}\in\mathbb{H}^{N\times N}$\\
 \hline
 ${\bm q}^H{\bm A}{\bm q}^*$ & $-\frac{1}{2}{\bm q}^H{\bm A}-\frac{1}{2}({\bm A}{\bm q}^*)^H$
 &${\bm q}^H{\bm A}+\mathfrak{R}(({\bm A}{\bm q}^*)^T)$& $\forall\bm{A}\in\mathbb{H}^{N\times N}$ \\
  \hline
 ${\bm q}^T{\bm A}{\bm q}^*$  & $-\frac{1}{2}{\bm q}^T{\bm A}+\mathfrak{R}(({\bm A}{\bm q}^*)^T)$
 &${\bm q}^T{\bm A}-\frac{1}{2}({\bm A}{\bm q}^*)^H$ & $\forall\bm{A}\in\mathbb{H}^{N\times N}$\\
 \hline
 ${\bm q}^T{\bm A}{\bm q}^*$  & $\frac{1}{2}({\bm q}^T{\bm A})^*$
 &$\frac{1}{2}{\bm q}^T{\bm A}$ & ${\bm A}^H={\bm A}$\\
 \hline
 ${\bm q}^H{\bm A}{\bm q}$  & ${\bm q}^H{\bm A}-\frac{1}{2}(\bm{Aq})^H$
 &$-\frac{1}{2}{\bm q}^H{\bm A}+\mathfrak{R}((\bm{Aq})^T)$ & $\forall\bm{A}\in\mathbb{H}^{N\times N}$\\
  \hline
${\bm q}^H{\bm A}{\bm q}$  & $\frac{1}{2}{\bm q}^H{\bm A}$
 &$-\frac{1}{2}({\bm q}^H{\bm A})^*$ & ${\bm A}^H={\bm A}$\\
  \hline
  ${\bm A}{\bm q}\beta$ &${\bm A}\mathfrak{R}(\beta)$ &$-\frac{1}{2}{\bm A}\beta^*$ & $\forall \bm{A}\in\mathbb{H}^{N\times N},\beta\in\mathbb{H}$\\
 \hline
 ${\bm A}{\bm q}^*\beta$  & $-\frac{1}{2}{\bm A}\beta^*$  &${\bm A}\mathfrak{R}(\beta)$ & $\forall\bm{A}\in\mathbb{H}^{N\times N},\beta\in\mathbb{H}$\\
 \hline
$\alpha{\bm q}^T{\bm A}$  & $\alpha\mathfrak{R}({\bm A}^T)$  &$-\frac{1}{2}\alpha{\bm A}^H$ & $\forall\bm{A}\in\mathbb{H}^{N\times N},\alpha\in\mathbb{H}$\\
 \hline
$\alpha{\bm q}^H{\bm A}$  & $-\frac{1}{2}\alpha{\bm A}^H$  &$\alpha\mathfrak{R}({\bm A}^T)$ & $\forall\bm{A}\in\mathbb{H}^{N\times N},\alpha\in\mathbb{H}$\\
 \hline
\end{tabular}\\[6pt]
\end{table*}

\section{Devolopment of quaternion derivatives}

\subsection{Derivatives of $f(q)$}
The case of scalar function of scalar variables is studied thoroughly in \cite{DPXU}.
In this case, the derivatives $\mathcal{D}_{q}f$ and $\mathcal{D}_{q}f$ become $\frac{\partial f}{\partial q}$ and $\frac{\partial f}{\partial q}$, respectively.
Some results of such functions are collected in Table \ref{tb:fderiv}, assuming $\alpha$, $\beta$ and $\lambda$ to be quaternion constants and $q$ to be a quaternion-valued variable.

\subsection{Derivatives of $f({\bm q})$}
Let $f:\mathbb{H}^{N\times 1}\rightarrow \mathbb{H}$ be $f({\bm q})={\bm q}^H{\bm A}{\bm q}$.
This kind of function frequently appears in quaternion filter optimization \cite{Took09,Took10,Took102} and array signal processing \cite{Bihan}.
For example, the optimization of the output power ${\bm q}^H{\bm A}{\bm q}$, where ${\bm q}$ is the filter coefficients and
${\bm A}$ is the input covariance matrix. i.e., ${\bm A}^H={\bm A}$. Using the product rule in Theorem \ref{thm:prodrl}, we have
$\mathcal{D}_{\bm q}{f}={\bm q}^H\mathcal{D}_{\bm q}{({\bm A}{\bm q})}+\mathcal{D}_{\bm q}{({\bm q}^H{\bm A}{\bm q})}|_{\bm{Aq}=const}={\bm q}^H{\bm A}-\frac{1}{2}(\bm{Aq})^H=\frac{1}{2}{\bm q}^H{\bm A}$.
The derivative
$\mathcal{D}_{{\bm q}^*}{\bm f}=\frac{1}{2}({\bm q}^H{\bm A})^*$ can be obtained in a similar manner.
Some results of such functions are shown in Table \ref{tb:derivfbmf}, assuming ${\bm a}\in\mathbb{H}^{N\times 1}$, ${\bm A}\in \mathbb{H}^{N\times N}$ to be constant.
and ${\bm q}\in \mathbb{H}^{N\times 1}$ to be vector variable.

\textit{1) Quaternion Least Mean Square:}
This section derive the quaternion least mean square (QLMS) algorithm given in \cite{Took09,Mandic11} using the GHR calculus.
The cost function to be minimized is a real-valued function
\begin{equation}
J(n)=|e(n)|^2=e^*(n)e(n)
\end{equation}
where
\begin{equation}\label{eq:qlmy}
e(n)=d(n)-y(n),\quad y(n)={\bm w}^T(n){\bm x}(n)
\end{equation}
and $d(n),y(n)\in \mathbb{H},{\bm w}(n),{\bm x}(n)\in \mathbb{H}^{N\times 1}$. The weight update of QLMS is then given by
\begin{equation}\label{eq:qlmsllearnrule}
{\bm w}(n+1)={\bm w}(n)-\eta(\mathcal{D}_{{\bm w}}J(n))^H
\end{equation}
where $\eta$ is the step size and $(\mathcal{D}_{{\bm w}}J(n))^H$ is the gradient of $J(n)$ with respect to ${\bm w}^*$, which defines the direction of the
maximum rate of change of $f$ from Theorem \ref{thm:extchang}. Using the results in Table \ref{tb:derivfbmf}, the gradient can be calculated by
\begin{align}\label{eq:qlmsdJdwconj1b}
\begin{split}
\mathcal{D}_{{\bm w}}J&=\mathcal{D}_{{\bm w}}((d-{\bm w}^T{\bm x})^*(d-{\bm w}^T{\bm x}))\\
&=\mathcal{D}_{{\bm w}}(d^*d)
-\mathcal{D}_{{\bm w}}(d^*{\bm w}^T{\bm x})\\
&\quad-\mathcal{D}_{{\bm w}}({\bm x}^H{\bm w}^*d)+\mathcal{D}_{{\bm w}}\left({\bm x}^H{\bm w}^*{\bm w}^T{\bm x}\right)\\
&=-d^*(n)\mathfrak{R}({\bm x}^T)+\frac{1}{2}{\bm x}^Hd^*\\
&\quad-\frac{1}{2}{\bm x}^H({\bm w}^T{\bm x})^*+{\bm x}^H{\bm w}^*\mathfrak{R}({\bm x}^T)\\
&=-\frac{1}{2}{\bm x}^Td^*+\frac{1}{2}{\bm x}^T({\bm w}^T{\bm x})^*\\
&=-\frac{1}{2}{\bm x}^T(d-{\bm w}^T{\bm x})^*=-\frac{1}{2}{\bm x}^Te^*
\end{split}
\end{align}
where time index `$n$' is omitted for convenience. Then, the update rules of QLMS becomes
\begin{equation}\label{eq:qlmsllearnrule1b}
{\bm w}(n+1)={\bm w}(n)+\eta\, e(n){\bm x}^*(n)
\end{equation}
where the constant $\frac{1}{2}$ in \eqref{eq:qlmsdJdwconj1b} is absorbed into $\eta$.
Note that if we start from $y(n)={\bf w}^H(n){\bf x}(n)$ rather than $y(n)={\bf w}^T(n){\bf x}(n)$ in \eqref{eq:qlmy}, then the final update rule of QLMS would become
\begin{equation}
{\bf w}(n+1)={\bf w}(n)+\eta\, {\bf x}(n)e^*(n)
\end{equation}
The QLMS algorithm in \eqref{eq:qlmsllearnrule1b} is a generic generalization of complex-valued LMS \cite{Widrow} to the case of quaternion vector.

\textit{2) Quaternion Widely Linear Least Mean Square:}
This section derive the widely linear QLMS (WL-QLMS) algorithm based on quaternion widely linear model given in \cite{Took11,Moreno12,Via10}. The cost function
to be minimized is
\begin{equation}
J(n)=|e(n)|^2=e^*(n)e(n)
\end{equation}
where
\begin{equation}
e(n)=d(n)-y(n),\quad e^*(n)=d^*(n)-y^*(n)
\end{equation}
and
\begin{equation}\label{eq:nwwlqlmsoutput}
\begin{split}
y(n)&={\bm h}^H(n){\bm x}(n)+{\bm g}^H(n){\bm x}^i(n)\\
&\quad+{\bm u}^H(n){\bm x}^j(n)+{\bm v}^H(n){\bm x}^k(n)
\end{split}
\end{equation}
The weight updates are then given by
\begin{equation}\label{eq:nwwlqlmsllearnrule}
\begin{split}
&{\bm h}(n+1)={\bm h}(n)-\eta (\mathcal{D}_{{\bm h}}J(n))^H\\
&{\bm g}(n+1)={\bm g}(n)-\eta (\mathcal{D}_{{\bm g}}J(n))^H\\
&{\bm u}(n+1)={\bm u}(n)-\eta (\mathcal{D}_{{\bm u}}J(n))^H\\
&{\bm v}(n+1)={\bm v}(n)-\eta (\mathcal{D}_{{\bm v}}J(n))^H
\end{split}
\end{equation}
where $\eta$ is the step size, $(\mathcal{D}_{{\bm h}}J(n))^H$, $(\mathcal{D}_{{\bm g}}J(n))^H$, $(\mathcal{D}_{{\bm u}}J(n))^H$ and $(\mathcal{D}_{{\bm v}}J(n))^H$ are the gradients of $J(n)$ with respect to ${\bm h}^*$, ${\bm g}^*$, ${\bm u}^*$ and ${\bm v}^*$, respectively.
Using the product rule in Theorem \ref{thm:prodrl}, the derivative $\mathcal{D}_{{\bm h}}J(n)$ is calculated by
\begin{align}\label{eq:qlmsdJdwconj2b}
\mathcal{D}_{{\bm h}}J=e^* \mathcal{D}_{{\bm h}}(e)+\mathcal{D}_{{\bm h}}\left((d-y)^*e\right)\big|_{e=const}
\end{align}
where time index `$n$' is omitted to ease the expressions above. We next calculate the following two derivatives
\begin{equation}\label{eq:iqlmsdedwconj}
\mathcal{D}_{{\bm h}}(e)=\mathcal{D}_{{\bm h}}(d-y)=-\mathcal{D}_{{\bm h}}({\bm h}^H{\bm x})=\frac{1}{2}{\bm x}^H
\end{equation}
\begin{align}\label{eq:iqlmsdeconjdwconj}
\mathcal{D}_{{\bm h}}\left((d-y)^*e\right)\big|_{e=const}&=-\mathcal{D}_{{\bm h}}({\bm x}^H{\bm h}\,e)\big|_{e=const}\nonumber\\
&=-{\bm x}^H\mathfrak{R}(e)
\end{align}
where the terms $\mathcal{D}_{{\bm q}}({\bm q}^H{\bm a})$ and $\mathcal{D}_{{\bm q}}({\bm a}^T{\bm q}\beta)$ are given in Table \ref{tb:derivfbmf},
and are used in the last equalities above. Substituting \eqref{eq:iqlmsdedwconj} and \eqref{eq:iqlmsdeconjdwconj} into \eqref{eq:qlmsdJdwconj2b} yields
\begin{equation}\label{eq:iqlmsdJdwconj1b}
\mathcal{D}_{{\bm h}}J=\frac{1}{2}e^*{\bm x}^H-{\bm x}^H\mathfrak{R}(e)=-\frac{1}{2}e\,{\bm x}^H
\end{equation}
The derivatives $\mathcal{D}_{{\bm g}}J(n)$, $\mathcal{D}_{{\bm u}}J(n)$ and $\mathcal{D}_{{\bm v}}J(n)$  can be calculated in a similar way to \eqref{eq:iqlmsdJdwconj1b} and are given by
\begin{equation}\label{eq:nwwlqlmsdJdwconj1c}
\begin{split}
&\mathcal{D}_{{\bm h}}J=-\frac{1}{2}e\,{\bm x}^H,\quad \mathcal{D}_{{\bm u}}J=-\frac{1}{2}e\,({\bm x}^j)^H\\
&\mathcal{D}_{{\bm g}}J=-\frac{1}{2}e\,({\bm x}^i)^H,\quad \mathcal{D}_{{\bm v}}J=-\frac{1}{2}e\,({\bm x}^k)^H
\end{split}
\end{equation}
Finally, the update within WL-QLMS can be expressed as
\begin{equation}\label{eq:nwwlqlmsllearnruleww}
\begin{split}
&{\bm h}(n+1)={\bm h}(n)+\eta\, {\bm x}(n)e^*(n)\\
&{\bm g}(n+1)={\bm g}(n)+\eta\,{\bm x}^i(n)e^*(n)\\
&{\bm u}(n+1)={\bm u}(n)+\eta\, {\bm x}^j(n)e^*(n)\\
&{\bm v}(n+1)={\bm v}(n)+\eta\,{\bm x}^k(n)e^*(n)
\end{split}
\end{equation}
where the constant $\frac{1}{2}$ in \eqref{eq:nwwlqlmsdJdwconj1c} is absorbed into $\eta$.

\textit{3) Quaternion Affine Projection Algorithm:}
This section derive the quaternion affine projection algorithm (QAPA) given in \cite{Jahanchahi13} based on the GHR calculus.
The aim of QAPA is to minimise adaptively the squared Euclidean norm of the change in the weight vector ${\bm w}(n)\in \mathbb{H}^{N\times 1}$, that is
\begin{equation}\label{eq:APA}
\begin{split}
&\textrm{minimise }\; \|\Delta {\bm w}(n)\|^{2} = \|{\bm w}(n+1) - {\bm  w}(n)\|^{2}\\
&\textrm{subject to }\; {\bm d}^{T}(n) = {\bm w}^{H}(n+1){\bm Q}(n)
\end{split}
\end{equation}
where ${\bm d}(n) = [d(n), \ldots, d(n-S+1)]^{T}\in \mathbb{H}^{S\times 1}$ denotes the desired signal vector
and ${\bm Q}(n) = [{\bm q}(n), \ldots, {\bm q}(n-S+1)]\in\mathbb{H}^{N\times S}$ denotes the matrix of $S$ past input vectors.
Using the Lagrange multipliers, the constrained optimisation problem \eqref{eq:APA} can be solved by the following cost function
\begin{equation}\label{eq:Cost_APA}
\begin{split}
J(n) &= \|{\bm w}(n+1) - {\bm w}(n)\|^{2} \\
&\quad+ \mathfrak{R}\{({\bm d}^{T}(n)- {\bm w}^{H}(n+1){\bm Q}(n)){\bm\lambda}^{*}\}\\
&= ({\bm w}(n+1) - {\bm w}(n))^H({\bm w}(n+1) - {\bm w}(n)) \\
&\quad+ \frac{1}{2}({\bm d}^{T}(n)- {\bm w}^{H}(n+1){\bm Q}(n))\bm \lambda^{*}\\
&\quad+ \frac{1}{2}{\bm \lambda}^{T}({\bm d}^{*}(n)- {\bm Q}^H(n){\bm w}(n+1))
\end{split}
\end{equation}
where ${\bm \lambda}\in\mathbb{H}^{S\times 1}$ denotes the Lagrange multipliers. Using the results in Table \ref{tb:derivfbmf}, we have
\begin{align}\label{eq:APA_Deriv_1}
\mathcal{D}_{{\bm w}(n+1)}J(n)&=\frac{1}{2}({\bm w}(n+1) - {\bm w}(n))^H\nonumber\\
&\quad+\frac{1}{2}\left(\frac{1}{2}({\bm Q}(n){\bm \lambda}^*)^H-{\bm \lambda}^T{\bm Q}^H(n)\right)\\
&=\frac{1}{2}({\bm w}(n+1) - {\bm w}(n))^H-\frac{1}{4}{\bm \lambda}^T{\bm Q}^H(n)\nonumber
\end{align}
Setting \eqref{eq:APA_Deriv_1} to zero, the weight update of QAPA can be obtained as
\begin{equation}
{\bm w}(k+1) - {\bm w}(k) = \frac{1}{2}{\bm Q}(n)\,{\bm \lambda}^*
\label{eq:APA_Deriv_2}
\end{equation}
Using the fact that ${\bm e}^{T}(n) = {\bm d}^{T}(n) - {\bm y}^{T}(n) = ({\bm w}^{H}(n+1)-{\bm w}^{H}(n)){\bm Q}(n)$, and based on \eqref{eq:APA_Deriv_2}, $\bm\lambda$ can be solved as
\begin{equation}\label{eq:APA_Deriv_3}
{\bm \lambda}^T = 2 {\bm e}^T(n)\left({\bm Q}^{H}(n){\bm Q}(n)\right)^{-1}
\end{equation}
which gives
\begin{equation}\label{eq:APA_Deriv_4}
{\bm w}(n+1) =  {\bm w}(n) + {\bm Q}(n)\left({\bm Q}^{H}(n){\bm Q}(n)\right)^{-1}{\bm e}^*(n)
\end{equation}
Note that to prevent the normalisation matrix ${\bm Q}^{H}(n){\bm Q}(n)$ within \eqref{eq:APA_Deriv_4} from becoming singular,
a small regularisation term $\varepsilon{\bm I}\in\mathbb{H}^{S\times S}$ is practically added with ${\bm I}$ the identity matrix,
whereas a step size $\eta$ is also incorporated to control the convergence and the steady state performance, giving the final weight update of QAPA in the form
\begin{eqnarray}\label{eq:APA_Update}
{\bm w}(n+1)= {\bm w}(n) +\eta\,{\bm Q}(n)\left({\bm Q}^{H}(n){\bm Q}(n)+\varepsilon{\bm I}\right)^{-1}{\bm e}^*(n)
\end{eqnarray}

\begin{table*}[!htbp]
\centering
\caption{Derivatives of functions of the type $f({\bm Q})$}\label{tb:derivtrQ}
\renewcommand{\arraystretch}{1.3} 
\begin{tabular}{|c|c|c|c|}
 \hline
$f({\bm Q})$   & $\frac{\partial f}{\partial {\bm Q}}$ & $\frac{\partial f}{\partial {\bm Q}^*}$ \\
 \hline
 \multicolumn{3}{c}{}\vspace{-6pt} \\
\hline
 $\textrm{Tr}({\bm Q})$ & ${\bm I}_N$  &$-\frac{1}{2}{\bm I}_N$ \\
 \hline
 $\textrm{Tr}({\bm Q}^H)$ & $-\frac{1}{2}{\bm I}_N$  &${\bm I}_N$ \\
 \hline
 $\textrm{Tr}({\bm A}{\bm Q})$ & ${\bm A}^T$  &$-\frac{1}{2}{\bm A}^T$ \\
 \hline
  $\textrm{Tr}({\bm A}{\bm Q}^H)$ & $-\frac{1}{2}{\bm A}$  &${\bm A}$ \\
 \hline
  $\textrm{Tr}({\bm Q}{\bm A})$  & $\mathfrak{R}({\bm A}^T)$  &$-\frac{1}{2}{\bm A}^H$ \\
 \hline
 $\textrm{Tr}({\bm Q}^H{\bm A})$ & $-\frac{1}{2}{\bm A}^*$  &$\mathfrak{R}({\bm A})$ \\
 \hline
 $\textrm{Tr}({\bm A}_1{\bm Q}{\bm A}_2)$  & ${\bm A}^T_1\mathfrak{R}({\bm A}^T_2)$  &$-\frac{1}{2}{\bm A}^T_1{\bm A}^H_2$ \\
 \hline
   $\textrm{Tr}({\bm A}_1{\bm Q}^*{\bm A}_2)$ & $-\frac{1}{2}{\bm A}^T_1{\bm A}^H_2$  &${\bm A}^T_1\mathfrak{R}({\bm A}^T_2)$ \\
 \hline
  $\textrm{Tr}({\bm A}_1{\bm Q}^T{\bm A}_2)$ & $\mathfrak{R}({\bm A}_2){\bm A}_1$  &$-\frac{1}{2}({\bm A}^T_1{\bm A}^H_2)^T$ \\
 \hline
 $\textrm{Tr}({\bm A}_1{\bm Q}^H{\bm A}_2)$  & $-\frac{1}{2}({\bm A}^T_1{\bm A}^H_2)^T$  &$\mathfrak{R}({\bm A}_2){\bm A}_1$ \\
 \hline
 $\textrm{Tr}({\bm Q}^n)$
 & $\sum\limits_{m=1}^n({\bm Q}^T)^{n-m}\mathfrak{R}({\bm Q}^{m-1})^T$
 &$-\frac{1}{2}\sum\limits_{m=1}^n({\bm Q}^T)^{n-m}({\bm Q}^{m-1})^H$ \\
 \hline
 $\textrm{Tr}({\bm Q}^{-1})$
 & $-({\bm Q}^T)^{-1}\mathfrak{R}({\bm Q}^{-1})^T$
 &$\frac{1}{2}({\bm Q}^T)^{-1}({\bm Q}^{-1})^H$ \\
 \hline
  $\textrm{Tr}({\bm A}_1{\bm Q}{\bm A}_2{\bm Q}{\bm A}_3)$
 & ${\bm A}^T_1\mathfrak{R}({\bm A}_2{\bm Q}{\bm A}_3)^T+({\bm A}_1{\bm Q}{\bm A}_2)^T\mathfrak{R}({\bm A}^T_3)$
 &$-\frac{1}{2}{\bm A}^T_1({\bm A}_2{\bm Q}{\bm A}_3)^H-\frac{1}{2}({\bm A}_1{\bm Q}{\bm A}_2)^T{\bm A}_3^H$ \\
 \hline
 $\textrm{Tr}({\bm A}_1{\bm Q}{\bm A}_2{\bm Q}^*{\bm A}_3)$
 & ${\bm A}^T_1\mathfrak{R}({\bm A}_2{\bm Q}^*{\bm A}_3)^T-\frac{1}{2}({\bm A}_1{\bm Q}{\bm A}_2)^T{\bm A}_3^H$
 &$-\frac{1}{2}{\bm A}^T_1({\bm A}_2{\bm Q}^*{\bm A}_3)^H+({\bm A}_1{\bm Q}{\bm A}_2)^T\mathfrak{R}({\bm A}^T_3)$ \\
 \hline
 $\textrm{Tr}({\bm A}_1{\bm Q}{\bm A}_2{\bm Q}^T{\bm A}_3)$
 & ${\bm A}^T_1\mathfrak{R}({\bm A}_2{\bm Q}^T{\bm A}_3)^T+\mathfrak{R}({\bm A}_3){\bm A}_1{\bm Q}{\bm A}_2$
 &$-\frac{1}{2}{\bm A}^T_1({\bm A}_2{\bm Q}^T{\bm A}_3)^H-\frac{1}{2}(({\bm A}_1{\bm Q}{\bm A}_2)^T{\bm A}_3^H)^T$ \\
 \hline
  $\textrm{Tr}({\bm A}_1{\bm Q}{\bm A}_2{\bm Q}^H{\bm A}_3)$
 & ${\bm A}^T_1\mathfrak{R}({\bm A}_2{\bm Q}^H{\bm A}_3)^T-\frac{1}{2}(({\bm A}_1{\bm Q}{\bm A}_2)^T{\bm A}_3^H)^T$
 &$-\frac{1}{2}{\bm A}^T_1({\bm A}_2{\bm Q}^H{\bm A}_3)^H+\mathfrak{R}({\bm A}_3){\bm A}_1{\bm Q}{\bm A}_2$ \\
 \hline
 $\textrm{Tr}({\bm A}_1{\bm Q}^*{\bm A}_2{\bm Q}{\bm A}_3)$
 & $-\frac{1}{2}{\bm A}^T_1({\bm A}_2{\bm Q}{\bm A}_3)^H+({\bm A}_1{\bm Q}^*{\bm A}_2)^T\mathfrak{R}({\bm A}^T_3)$
 &${\bm A}^T_1\mathfrak{R}({\bm A}_2{\bm Q}{\bm A}_3)^T-\frac{1}{2}({\bm A}_1{\bm Q}^*{\bm A}_2)^T{\bm A}_3^H$ \\
 \hline
 $\textrm{Tr}({\bm A}_1{\bm Q}^*{\bm A}_2{\bm Q}^*{\bm A}_3)$
 & $-\frac{1}{2}{\bm A}^T_1({\bm A}_2{\bm Q}^*{\bm A}_3)^H-\frac{1}{2}({\bm A}_1{\bm Q}^*{\bm A}_2)^T{\bm A}_3^H$
 &${\bm A}^T_1\mathfrak{R}({\bm A}_2{\bm Q}^*{\bm A}_3)^T+({\bm A}_1{\bm Q}^*{\bm A}_2)^T\mathfrak{R}({\bm A}^T_3)$ \\
 \hline
 $\textrm{Tr}({\bm A}_1{\bm Q}^*{\bm A}_2{\bm Q}^T{\bm A}_3)$
 & $-\frac{1}{2}{\bm A}^T_1({\bm A}_2{\bm Q}^T{\bm A}_3)^H+\mathfrak{R}({\bm A}_3){\bm A}_1{\bm Q}^*{\bm A}_2$
 &${\bm A}^T_1\mathfrak{R}({\bm A}_2{\bm Q}^T{\bm A}_3)^T-\frac{1}{2}(({\bm A}_1{\bm Q}^*{\bm A}_2)^T{\bm A}_3^H)^T$ \\
 \hline
 $\textrm{Tr}({\bm A}_1{\bm Q}^*{\bm A}_2{\bm Q}^H{\bm A}_3)$
 & $-\frac{1}{2}{\bm A}^T_1({\bm A}_2{\bm Q}^H{\bm A}_3)^H-\frac{1}{2}(({\bm A}_1{\bm Q}^*{\bm A}_2)^T{\bm A}_3^H)^T$
 &${\bm A}^T_1\mathfrak{R}({\bm A}_2{\bm Q}^H{\bm A}_3)^T+\mathfrak{R}({\bm A}_3){\bm A}_1{\bm Q}^*{\bm A}_2$ \\
 \hline
  $\textrm{Tr}({\bm A}_1{\bm Q}^T{\bm A}_2{\bm Q}{\bm A}_3)$
 & $\mathfrak{R}({\bm A}_2{\bm Q}{\bm A}_3){\bm A}_1+({\bm A}_1{\bm Q}^T{\bm A}_2)^T\mathfrak{R}({\bm A}^T_3)$
 &$-\frac{1}{2}({\bm A}^T_1({\bm A}_2{\bm Q}{\bm A}_3)^H)^T-\frac{1}{2}({\bm A}_1{\bm Q}^T{\bm A}_2)^T{\bm A}_3^H$ \\
 \hline
   $\textrm{Tr}({\bm A}_1{\bm Q}^T{\bm A}_2{\bm Q}^*{\bm A}_3)$
 & $\mathfrak{R}({\bm A}_2{\bm Q}^*{\bm A}_3){\bm A}_1-\frac{1}{2}({\bm A}_1{\bm Q}^T{\bm A}_2)^T{\bm A}_3^H$
 &$-\frac{1}{2}({\bm A}^T_1({\bm A}_2{\bm Q}^*{\bm A}_3)^H)^T+({\bm A}_1{\bm Q}^T{\bm A}_2)^T\mathfrak{R}({\bm A}^T_3)$ \\
 \hline
    $\textrm{Tr}({\bm A}_1{\bm Q}^T{\bm A}_2{\bm Q}^T{\bm A}_3)$
 & $\mathfrak{R}({\bm A}_2{\bm Q}^T{\bm A}_3){\bm A}_1+\mathfrak{R}({\bm A}_3){\bm A}_1{\bm Q}^T{\bm A}_2$
 &$-\frac{1}{2}({\bm A}^T_1({\bm A}_2{\bm Q}^T{\bm A}_3)^H)^T-\frac{1}{2}(({\bm A}_1{\bm Q}^T{\bm A}_2)^T{\bm A}_3^H)^T$ \\
 \hline
     $\textrm{Tr}({\bm A}_1{\bm Q}^T{\bm A}_2{\bm Q}^H{\bm A}_3)$
 & $\mathfrak{R}({\bm A}_2{\bm Q}^H{\bm A}_3){\bm A}_1-\frac{1}{2}(({\bm A}_1{\bm Q}^T{\bm A}_2)^T{\bm A}_3^H)^T$
 &$-\frac{1}{2}({\bm A}^T_1({\bm A}_2{\bm Q}^H{\bm A}_3)^H)^T+\mathfrak{R}({\bm A}_3){\bm A}_1{\bm Q}^T{\bm A}_2$ \\
 \hline
    $\textrm{Tr}({\bm A}_1{\bm Q}^H{\bm A}_2{\bm Q}{\bm A}_3)$
 & $-\frac{1}{2}({\bm A}^T_1({\bm A}_2{\bm Q}{\bm A}_3)^H)^T+({\bm A}_1{\bm Q}^H{\bm A}_2)^T\mathfrak{R}({\bm A}^T_3)$
 &$\mathfrak{R}({\bm A}_2{\bm Q}{\bm A}_3){\bm A}_1-\frac{1}{2}({\bm A}_1{\bm Q}^H{\bm A}_2)^T{\bm A}_3^H$ \\
 \hline
    $\textrm{Tr}({\bm A}_1{\bm Q}^H{\bm A}_2{\bm Q}^*{\bm A}_3)$
 & $-\frac{1}{2}({\bm A}^T_1({\bm A}_2{\bm Q}^*{\bm A}_3)^H)^T-\frac{1}{2}({\bm A}_1{\bm Q}^H{\bm A}_2)^T{\bm A}_3^H$
 &$\mathfrak{R}({\bm A}_2{\bm Q}^*{\bm A}_3){\bm A}_1+({\bm A}_1{\bm Q}^H{\bm A}_2)^T\mathfrak{R}({\bm A}^T_3)$ \\
 \hline
     $\textrm{Tr}({\bm A}_1{\bm Q}^H{\bm A}_2{\bm Q}^T{\bm A}_3)$
 & $-\frac{1}{2}({\bm A}^T_1({\bm A}_2{\bm Q}^T{\bm A}_3)^H)^T+\mathfrak{R}({\bm A}_3){\bm A}_1{\bm Q}^H{\bm A}_2$
 &$\mathfrak{R}({\bm A}_2{\bm Q}^T{\bm A}_3){\bm A}_1-\frac{1}{2}(({\bm A}_1{\bm Q}^H{\bm A}_2)^T{\bm A}_3^H)^T$ \\
 \hline
      $\textrm{Tr}({\bm A}_1{\bm Q}^H{\bm A}_2{\bm Q}^H{\bm A}_3)$
 & $-\frac{1}{2}({\bm A}^T_1({\bm A}_2{\bm Q}^H{\bm A}_3)^H)^T-\frac{1}{2}(({\bm A}_1{\bm Q}^H{\bm A}_2)^T{\bm A}_3^H)^T$
 &$\mathfrak{R}({\bm A}_2{\bm Q}^H{\bm A}_3){\bm A}_1+\mathfrak{R}({\bm A}_3){\bm A}_1{\bm Q}^H{\bm A}_2$ \\
 \hline
\end{tabular}\\[6pt]
\end{table*}

\subsection{Derivatives of $f({\bm Q})$}
For scalar functions $f: \mathbb{H}^{N\times S}\rightarrow\mathbb{H}$,
it is common to define the following notations of matrix derivative
\begin{align}
\frac{\partial f}{\partial {\bm Q}^{\mu}}\triangleq\left[\begin{array}{ccc}
\frac{\partial f}{\partial q^{\mu}_{11}}& \cdots &\frac{\partial f}{\partial q^{\mu}_{1S}}\\
\vdots  & \ddots & \vdots \\
\frac{\partial f}{\partial q^{\mu}_{N1}}& \cdots &\frac{\partial f}{\partial q^{\mu}_{NS}}
\end{array}\right]\label{def:matnew}\\
\frac{\partial f}{\partial {\bm Q}^{\mu*}}\triangleq\left[\begin{array}{ccc}
\frac{\partial f}{\partial q^{\mu*}_{11}}& \cdots &\frac{\partial f}{\partial q^{\mu*}_{1S}}\\
\vdots  & \ddots & \vdots \\
\frac{\partial f}{\partial q^{\mu*}_{N1}}& \cdots &\frac{\partial f}{\partial q^{\mu*}_{NS}}
\end{array}\right]\label{def:matnew2}
\end{align}
which are called the gradient of $f$ with respect to ${\bm Q}^{\mu}$ and ${\bm Q}^{\mu*}$.
Equations \eqref{def:matnew} and \eqref{def:matnew2} are generalizations of the real- and complex-valued
case given in \cite{Magnus,Hjorungnes07} to the quaternion case. By comparing \eqref{def:matderi}  and \eqref{def:matderi2} with \eqref{def:matnew} and \eqref{def:matnew},
their connection is given by
\begin{equation}
\mathcal{D}_{{\bm Q}^{\mu}}f=\textrm{vec}^T\left(\frac{\partial f}{\partial {\bm Q}^{\mu}}\right),\quad
\mathcal{D}_{{\bm Q}^{\mu*}}f=\textrm{vec}^T\left(\frac{\partial f}{\partial {\bm Q}^{\mu*}}\right)
\end{equation}
Then, the steepest descent method \eqref{eq:stdesc} can be reformulated as
\begin{equation}\label{eq:stdesc}
{\bm Q}_{n+1}={\bm Q}_{n}-\eta\frac{\partial f}{\partial {\bm Q}^{*}},\quad f\in\mathbb{R}
\end{equation}
where $\eta>0$ is the step size. Some important results of functions of the type $f({\bm Q})$ are summarized in
Table \ref{tb:derivtrQ}, where ${\bm Q}\in\mathbb{H}^{N\times S}$ or possibly ${\bm Q}\in\mathbb{H}^{N\times N}$ for the functions to be defined, and ${\bm A}_1,{\bm A}_2,{\bm A}_3,{\bm A}$ are chosen such that the functions are well defined.

\textit{1) Quaternion Matrix Least Squares:} Given ${\bm A}\in \mathbb{H}^{R\times N}$, ${\bm B}\in \mathbb{H}^{S\times P}$ and ${\bm
C}\in \mathbb{H}^{R\times P}$, find ${\bm Q}\in \mathbb{H}^{N\times S}$ such that
 the error of the overdetermined linear system of equations
\begin{equation}\label{eq:qlsproblem}
F({\bm Q})={\rm Tr}\left\{\left({\bm C}-\bm{AQB}\right)^H\left({\bm C}-\bm{AQB}\right)\right\}
\end{equation}
is minimized. Using the results in Table \ref{tb:derivtrQ}, the gradient of $F({\bm Q})$ can be derived as
\begin{align}\label{eq:lqsgrad}
\frac{\partial F({\bm Q})}{\partial {\bm Q}^*}&=\frac{\partial {\rm Tr}\left\{{\bm C}^H{\bm C}-{\bm C}^H\bm{AQB}-{\bm B}^H{\bm Q}^H{\bm A}^H{\bm C} \right\}}{\partial {\bm Q}^*}\nonumber\\
&\quad +\frac{\partial {\rm Tr}\left\{{\bm B}^H{\bm Q}^H{\bm A}^H\bm{AQB} \right\}}{\partial {\bm Q}^*}\nonumber\\
&=\frac{1}{2}({\bm C}^H{\bm A})^T{\bm B}^H-\mathfrak{R}({\bm A}^H{\bm C}){\bm B}^H\nonumber\\
&\quad+\mathfrak{R}({\bm A}^H\bm{AQB}){\bm B}^H-\frac{1}{2}({\bm B}^H{\bm Q}^H{\bm A}^H{\bm A})^T{\bm B}^H\nonumber\\
&=-\frac{1}{2}({\bm A}^H{\bm C}){\bm B}^H+\frac{1}{2}({\bm A}^H\bm{AQB}){\bm B}^H\nonumber\\
&=-\frac{1}{2}{\bm A}^H({\bm C}-\bm{AQB}){\bm B}^H
\end{align}
Setting \eqref{eq:lqsgrad} to be zero, we obtain a normal equation
\begin{equation}
{\bm A}^H\bm{AQB}{\bm B}^H={\bm A}^H{\bm C}{\bm B}^H
\end{equation}
If ${\bm A}^H{\bm A}$ and ${\bm B}{\bm B}^H$ are invertible, then the system \eqref{eq:qlsproblem} has a unique solution
\begin{equation}
{\bm Q}=({\bm A}^H{\bm A})^{-1}{\bm A}^H{\bm C}{\bm B}^H({\bm B}{\bm B}^H)^{-1}
\end{equation}

\subsection{Derivatives of ${\bm F}({\bm Q})$}
We next present the derivatives of some elementary matrix functions, which are often used in nonlinear adaptive filters and neural networks.
For other useful examples of matrix functions, we simply apply the provided
theory and summarize the results in Table \ref{tb:matderiv}, where ${\bm Q}\in\mathbb{H}^{N\times S}$ or possibly ${\bm Q}\in\mathbb{H}^{N\times N}$ for the functions to be defined,
and ${\bm A}_1,{\bm A}_2,{\bm A}$ are chosen such that the functions are well defined.

\textit{1) Derivatives of Power Function:} Let ${\bm F}:\mathbb{H}^{N\times N}\rightarrow \mathbb{H}^{N\times N}$ be given by
${\bm F}({\bm Q})={\bm Q}^n$, where $n$ is a positive integer number. Using the product rule in Theorem \ref{thm:prodrl}, we have
\begin{flalign}
\mathcal{D}_{{\bm Q}}({\bm Q}^n)&=({\bm I}_N\otimes {\bm Q})\mathcal{D}_{{\bm Q}}({{\bm Q}^{n-1}})+\mathcal{D}_{{\bm Q}}({\bm Q}{\bm Q}^{n-1})|_{{\bm Q}^{n-1}=const}\nonumber\\
&=({\bm I}_N\otimes {\bm Q})\mathcal{D}_{{\bm Q}}({{\bm Q}^{n-1}})+\mathfrak{R}({\bm Q}^{n-1})^T\otimes {\bm I}_N
\end{flalign}
where the term $\mathcal{D}_{{\bm Q}}({\bm Q}{\bm A})$, given in Table \ref{tb:matderiv}, was used in the last equality.
Note that the above expression is recurrent about $\mathcal{D}_{{\bm Q}}({\bm Q}^n)$.
Expanding this expression and using the initial condition $\mathcal{D}_{{\bm Q}}({\bm Q})={\bm I}_{N^2}$, yields
\begin{equation}\label{eq:res1powfun}
\mathcal{D}_{{\bm Q}}({\bm Q}^n)=\sum_{m=1}^n({\bm I}_N\otimes {\bm Q})^{n-m}(\mathfrak{R}({\bm Q}^{m-1})^T\otimes {\bm I}_N)
\end{equation}
In a similar manner, we have
\begin{equation}
\mathcal{D}_{{\bm Q}^*}({\bm Q}^n)=-\frac{1}{2}\sum_{m=1}^n({\bm I}_N\otimes {\bm Q})^{n-m}(({\bm Q}^{m-1})^H\otimes {\bm I}_N)
\end{equation}

\textit{2) Derivatives of Exponential Function:} Let ${\bm F}:\mathbb{H}^{N\times N}\rightarrow \mathbb{H}^{N\times N}$ be given by
${\bm F}({\bm Q})=\sum\limits_{n=0}^{+\infty}\frac{{\bm Q}^n}{n!}$. From \eqref{eq:res1powfun}, we have
\begin{equation}\label{eq:expderivative1}
\mathcal{D}_{{\bm Q}}{\bm F}=\sum_{n=0}^{+\infty}\sum_{m=1}^n\frac{1}{n!}({\bm I}_N\otimes {\bm Q})^{n-m}(\mathfrak{R}({\bm Q}^{m-1})^T\otimes {\bm I}_N)
\end{equation}
In a similar manner, we have
\begin{equation}\label{eq:expderivative5}
\mathcal{D}_{{\bm Q}^*}{\bm F}=\sum_{n=0}^{+\infty}\sum_{m=1}^n\frac{-1}{2\,n!}({\bm I}_N\otimes {\bm Q})^{n-m}(({\bm Q}^{m-1})^H\otimes {\bm I}_N)
\end{equation}
The two examples are a generation of the quaternion scalar variable case treated in \cite{DPXU} to the quaternion matrix variable case. Likewise,
the derivatives of the  trigonometric functions and hyperbolic functions can be derived in terms of the exponential
function.

\begin{table*}[!htbp]
\centering
\caption{Derivatives of functions of the type ${\bm F}({\bm Q})$}\label{tb:matderiv}
\renewcommand{\arraystretch}{1.3} 
\begin{tabular}{|c|c|c|c|}
 \hline
${\bm F}({\bm Q})$   & $\mathcal{D}_{\bm Q}{\bm F}$ & $\mathcal{D}_{{\bm Q}^*}{\bm F}$ \\
 \hline
 \multicolumn{3}{c}{}\vspace{-6pt} \\
\hline
 ${\bm Q}$ & ${\bm I}_{NS}$  &$-\frac{1}{2}{\bm I}_{NS}$ \\
 \hline
 ${\bm Q}^H$ & $-\frac{1}{2}{\bm K}_{N,S}$  &${\bm K}_{N,S}$ \\
 \hline
 ${\bm A}{\bm Q}$ & ${\bm I}_S \otimes{\bm A}$  &$-\frac{1}{2}{\bm I}_S \otimes{\bm A}$ \\
 \hline
${\bm A}{\bm Q}^*$ & $-\frac{1}{2}{\bm I}_S \otimes{\bm A}$  &${\bm I}_S \otimes{\bm A}$ \\
 \hline
${\bm A}{\bm Q}^T$ & $({\bm I}_N \otimes{\bm A}){\bm K}_{N,S}$  &$-\frac{1}{2}({\bm I}_N \otimes{\bm A}){\bm K}_{N,S}$ \\
 \hline
  ${\bm A}{\bm Q}^H$ & $-\frac{1}{2}({\bm I}_N \otimes{\bm A}){\bm K}_{N,S}$  &$({\bm I}_N \otimes{\bm A}){\bm K}_{N,S}$ \\
 \hline
  ${\bm Q}{\bm A}$  & $\mathfrak{R}({\bm A}^T)\otimes {\bm I}_N$  &$-\frac{1}{2}{\bm A}^H\otimes {\bm I}_N$ \\
 \hline
 ${\bm Q}^*{\bm A}$  & $-\frac{1}{2}{\bm A}^H\otimes {\bm I}_N$  &$\mathfrak{R}({\bm A}^T)\otimes {\bm I}_N$ \\
 \hline
  ${\bm Q}^T{\bm A}$ & $(\mathfrak{R}({\bm A}^T)\otimes {\bm I}_S){\bm K}_{N,S}$  &$-\frac{1}{2}({\bm A}^H\otimes {\bm I}_S){\bm K}_{N,S}$ \\
 \hline
 ${\bm Q}^H{\bm A}$ & $-\frac{1}{2}({\bm A}^H\otimes {\bm I}_S){\bm K}_{N,S}$  &$(\mathfrak{R}({\bm A}^T)\otimes {\bm I}_S){\bm K}_{N,S}$ \\
 \hline
 ${\bm A}_1{\bm Q}{\bm A}_2$  & $({\bm I}_P\otimes{\bm A}_1)(\mathfrak{R}({\bm A}^T_2)\otimes {\bm I}_N)$  &$-\frac{1}{2}({\bm I}_P\otimes{\bm A}_1)({\bm A}^H_2\otimes {\bm I}_N)$ \\
 \hline
   ${\bm A}_1{\bm Q}^*{\bm A}_2$ & $-\frac{1}{2}({\bm I}_P\otimes{\bm A}_1)({\bm A}^H_2\otimes {\bm I}_N)$  &$({\bm I}_P\otimes{\bm A}_1)(\mathfrak{R}({\bm A}^T_2)\otimes {\bm I}_N)$ \\
 \hline
  ${\bm A}_1{\bm Q}^T{\bm A}_2$ & $({\bm I}_P\otimes{\bm A}_1)(\mathfrak{R}({\bm A}^T_2)\otimes {\bm I}_S){\bm K}_{N,S}$  &$-\frac{1}{2}({\bm I}_P\otimes{\bm A}_1)({\bm A}^H_2\otimes {\bm I}_S){\bm K}_{N,S}$ \\
 \hline
 ${\bm A}_1{\bm Q}^H{\bm A}_2$  & $-\frac{1}{2}({\bm I}_P\otimes{\bm A}_1)({\bm A}^H_2\otimes {\bm I}_S){\bm K}_{N,S}$  &$({\bm I}_P\otimes{\bm A}_1)(\mathfrak{R}({\bm A}^T_2)\otimes {\bm I}_S){\bm K}_{N,S}$ \\
 \hline
 ${\bm Q}^n$
 & $\sum\limits_{m=1}^n({\bm I}_N\otimes {\bm Q})^{n-m}(\mathfrak{R}({\bm Q}^{m-1})^T\otimes {\bm I}_N)$
 &$-\frac{1}{2}\sum\limits_{m=1}^n({\bm I}_N\otimes {\bm Q})^{n-m}(({\bm Q}^{m-1})^H\otimes {\bm I}_N)$ \\
 \hline
 ${\bm Q}^{-1}$
 & $-({\bm I}_N\otimes{\bm Q})^{-1}(\mathfrak{R}({\bm Q}^{-1})^T\otimes {\bm I}_N)$
 &$\frac{1}{2}({\bm I}_N\otimes{\bm Q})^{-1}(({\bm Q}^{-1})^H\otimes {\bm I}_N)$ \\
 \hline
  ${\bm Q}{\bm A}{\bm Q}^T$
 & $\mathfrak{R}({\bm A}{\bm Q}^T)^T\otimes {\bm I}_N+({\bm I}_N \otimes({\bm Q}{\bm A})){\bm K}_{N,S}$
 &$-\frac{1}{2}({\bm A}{\bm Q}^T)^H\otimes {\bm I}_N-\frac{1}{2}({\bm I}_N \otimes(\bm{QA})){\bm K}_{N,S}$ \\
 \hline
  ${\bm Q}{\bm A}{\bm Q}^H$
 & $\mathfrak{R}({\bm A}{\bm Q}^H)^T\otimes {\bm I}_N-\frac{1}{2}({\bm I}_N \otimes({\bm Q}{\bm A})){\bm K}_{N,S}$
 &$-\frac{1}{2}({\bm A}{\bm Q}^H)^H\otimes {\bm I}_N+({\bm I}_N \otimes(\bm{QA})){\bm K}_{N,S}$ \\
 \hline
 ${\bm Q}^*{\bm A}{\bm Q}^T$
 & $-\frac{1}{2}({\bm A}{\bm Q}^T)^H\otimes {\bm I}_N+({\bm I}_N \otimes({\bm Q}^*{\bm A})){\bm K}_{N,S}$
 &$\mathfrak{R}({\bm A}{\bm Q}^T)^T\otimes {\bm I}_N-\frac{1}{2}({\bm I}_N \otimes({\bm Q}^*{\bm A})){\bm K}_{N,S}$ \\
 \hline
${\bm Q}^*{\bm A}{\bm Q}^H$
 & $-\frac{1}{2}({\bm A}{\bm Q}^H)^H\otimes {\bm I}_N-\frac{1}{2}({\bm I}_N \otimes({\bm Q}^*{\bm A})){\bm K}_{N,S}$
 &$\mathfrak{R}({\bm A}{\bm Q}^H)^T\otimes {\bm I}_N+({\bm I}_N \otimes({\bm Q}^*{\bm A})){\bm K}_{N,S}$ \\
 \hline
 ${\bm Q}^T{\bm A}{\bm Q}$
 & $(\mathfrak{R}(\bm{AQ})^T\otimes {\bm I}_S){\bm K}_{N,S}+{\bm I}_S \otimes({\bm Q}^T{\bm A})$
 &$-\frac{1}{2}((\bm{AQ})^H\otimes {\bm I}_S){\bm K}_{N,S}-\frac{1}{2}{\bm I}_S \otimes({\bm Q}^T{\bm A})$ \\
 \hline
  ${\bm Q}^T{\bm A}{\bm Q}^*$
 & $(\mathfrak{R}({\bm A}{\bm Q}^*)^T\otimes {\bm I}_S){\bm K}_{N,S}-\frac{1}{2}{\bm I}_S \otimes({\bm Q}^T{\bm A})$
 &$-\frac{1}{2}(({\bm A}{\bm Q}^*)^H\otimes {\bm I}_S){\bm K}_{N,S}+{\bm I}_S \otimes({\bm Q}^T{\bm A})$ \\
 \hline
 ${\bm Q}^H{\bm A}{\bm Q}$
 & $-\frac{1}{2}((\bm{AQ})^H\otimes {\bm I}_S){\bm K}_{N,S}+{\bm I}_S \otimes({\bm Q}^H{\bm A})$
 &$(\mathfrak{R}(\bm{AQ})^T\otimes {\bm I}_S){\bm K}_{N,S}-\frac{1}{2}{\bm I}_S \otimes({\bm Q}^H{\bm A})$ \\
 \hline
  ${\bm Q}^H{\bm A}{\bm Q}^*$
 & $-\frac{1}{2}({\bm A}{\bm Q}^*)^H\otimes {\bm I}_S){\bm K}_{N,S}-\frac{1}{2}{\bm I}_S \otimes({\bm Q}^H{\bm A})$
 &$(\mathfrak{R}({\bm A}{\bm Q}^*)^T\otimes {\bm I}_S){\bm K}_{N,S}+{\bm I}_S \otimes({\bm Q}^H{\bm A})$ \\
 \hline
\end{tabular}\\[6pt]
\end{table*}

\section{Conclusions}
A systematic framework for the calculation of derivatives of quaternion matrix functions of quaternion matrix variables has been proposed based on the GHR calculus.
New matrix forms of product and chain rules have been introduced to conveniently calculate the derivatives of quaternion matrix functions,
and several theorems have been developed for quaternion gradient optimisation, such as for the identification of stationary points, direction of
maximum change problems, and the steepest descent methods. Furthermore, the usefulness of the presented method
has been illustrated on some typical gradient based optimization
problems in signal processing. Key results are given in tabular form.

\section*{Acknowledgment}
The authors dedicate this work to the memory of Prof. Are Hj{\o}rungnes, a pioneer of complex-valued matrix derivatives.

%

\end{document}